\documentclass[12pt,a4paper,draft]{amsart}

\usepackage{latexsym}
\usepackage{ifthen}
\usepackage[leqno]{amsmath}
\usepackage{enumerate}
\usepackage{calc}
\usepackage{amstext,amsbsy,amsopn,amsthm,amsgen,amsfonts,amscd,amsxtra,upref}
\usepackage{mathrsfs}\usepackage{euscript}\usepackage{amssymb}

\swapnumbers
\theoremstyle{plain}
\newtheorem{Thm}{Theorem}[section]
\newtheorem{Lem}[Thm]{Lemma}
\newtheorem{Cor}[Thm]{Corollary}
\newtheorem{Pro}[Thm]{Proposition}
\newtheorem{Prp}[Thm]{Properties}
\newtheorem{Sub}[Thm]{Sublemma}

\theoremstyle{definition}
\newtheorem{Def}[Thm]{Definition}
\newtheorem{Exm}[Thm]{Example}
\newtheorem{Exs}[Thm]{Examples}

\theoremstyle{remark}
\newtheorem{Rem}[Thm]{Remark}
\newtheorem{Rms}[Thm]{Remarks}
\newtheorem*{Com}{Commentary}

\newcommand{\myEmail}{piotr.niemiec@uj.edu.pl}
\newcommand{\myAddress}{\noindent{}Piotr Niemiec\\{}Jagiellonian University\\{}Institute of Mathematics\\{}
   ul. \L{}ojasiewicza 6\\{}30-348 Krak\'{o}w\\{}Poland}
\newcommand{\myData}{\author[P. Niemiec]{Piotr Niemiec}\address{\myAddress}\email{\myEmail}}

\newcommand{\NNN}{\mathbb{N}}
\newcommand{\RRR}{\mathbb{R}}

\newcommand{\BBb}{\CMcal{B}}\newcommand{\CCc}{\CMcal{C}}\newcommand{\DDd}{\CMcal{D}}
\newcommand{\EEe}{\CMcal{E}}\newcommand{\FFf}{\CMcal{F}}\newcommand{\GGg}{\CMcal{G}}\newcommand{\HHh}{\CMcal{H}}
\newcommand{\KKk}{\CMcal{K}}\newcommand{\LLl}{\CMcal{L}}
\newcommand{\OOo}{\CMcal{O}}
\newcommand{\TTt}{\CMcal{T}}
\newcommand{\UUu}{\CMcal{U}}\newcommand{\VVv}{\CMcal{V}}\newcommand{\WWw}{\CMcal{W}}

\newcommand{\AaA}{\EuScript{A}}\newcommand{\CcC}{\EuScript{C}}
\newcommand{\DdD}{\EuScript{D}}\newcommand{\FfF}{\EuScript{F}}

\newcommand{\MmM}{\EuScript{M}}

\newcommand{\SsS}{\EuScript{S}}
\newcommand{\WwW}{\EuScript{W}}

\newcommand{\Dd}{\mathfrak{D}}

\newcommand{\mM}{\mathfrak{m}}

\newcommand{\SECT}[1]{\section{#1}\renewcommand{\theequation}{\thesection-\arabic{equation}}\setcounter{equation}{0}}

\newcounter{help}
\newcommand{\ITE}[3]{\ifthenelse{#1}{#2}{#3}}\newcommand{\ITEE}[3]{\ITE{\equal{#1}{#2}}{#3}{}}

\newcommand{\im}{\operatorname{im}}\newcommand{\id}{\operatorname{id}}
\newcommand{\intt}{\operatorname{int}}\newcommand{\lin}{\operatorname{lin}}
\newcommand{\card}{\operatorname{card}}

\newcommand{\sgn}{\operatorname{sgn}}
\newcommand{\ord}{\operatorname{ord}}
\newcommand{\cov}{\operatorname{cov}}
\newcommand{\comp}{\operatorname{comp}}
\newcommand{\mesh}{\operatorname{mesh}}\newcommand{\st}{\operatorname{st}}

\newcommand{\scalarr}{\langle\cdot,\mathrm{-}\rangle}
\newcommand{\scalar}[2]{\langle #1,#2\rangle}\newcommand{\leqsl}{\leqslant}\newcommand{\geqsl}{\geqslant}

\newcommand{\epsi}{\varepsilon}\newcommand{\varempty}{\varnothing}\newcommand{\dd}{\colon}

\newcommand{\inj}{\hookrightarrow}

\newcommand{\TFCAE}{The following conditions are equivalent:}
\newcommand{\tfcae}{the following conditions are equivalent:}\newcommand{\iaoi}{if and only if}

\newcommand{\THM}[1]{Theorem~\textup{\ref{thm:#1}}}
\newcommand{\COR}[1]{Corollary~\textup{\ref{cor:#1}}}
\newcommand{\LEM}[1]{Lemma~\textup{\ref{lem:#1}}}\newcommand{\PRO}[1]{Proposition~\textup{\ref{pro:#1}}}

\newcommand{\EXM}[1]{Example~\textup{\ref{exm:#1}}}

\newenvironment{thm}[1]{\begin{Thm}\label{thm:#1}}{\end{Thm}}\newenvironment{lem}[1]{\begin{Lem}\label{lem:#1}}{\end{Lem}}
\newenvironment{cor}[1]{\begin{Cor}\label{cor:#1}}{\end{Cor}}\newenvironment{pro}[1]{\begin{Pro}\label{pro:#1}}{\end{Pro}}
\newenvironment{dfn}[1]{\begin{Def}\label{def:#1}}{\end{Def}}
\newenvironment{exm}[1]{\begin{Exm}\label{exm:#1}}{\end{Exm}}
\newenvironment{rem}[1]{\begin{Rem}\label{rem:#1}}{\end{Rem}}

\newenvironment{lemm}[2]{\begin{Lem}[#2]\label{lem:#1}}{\end{Lem}}

\newcommand{\bibITEM}[2]{\ITE{\equal{#2}{}}{\bibitem{#1} }{\bibitem[#2]{#1} }}
\newcommand{\BIB}[8]{
   \bibITEM{#1}{#8} #2, \textit{#3}, #4{} \textbf{#5} (#6), #7.}
\newcommand{\myBIB}[6][P. Niemiec]{#1, \textit{#2}, #3{}\ITE{\equal{#4}{}}{}{ \textbf{#4}} (#5), #6.}
\newcommand{\BIb}[6]{
   \bibITEM{#1}{#6} #2, \textit{#3}, #4, #5.}
\newcommand{\BiB}[9]{
   \bibITEM{#1}{#9} #2, \textit{#3}, #4{} \textit{#5}, #6, #7, #8.}

\newcommand{\myBAPP}[3][P. Niemiec]{
   #1, \textit{#2}, #3}

\newcommand{\jRN}[2][]{
   \ITEE{#2}{ActaM}{\ITE{\equal{#1}{+}}
      {Acta Mathematica}{Acta Math.}}
   \ITEE{#2}{ActaMSinES}{\ITE{\equal{#1}{+}}
      {Acta Mathematica Sinica (English Series)}{Acta Math. Sin. (Engl. Ser.)}}
   \ITEE{#2}{AdvM}{\ITE{\equal{#1}{+}}
      {Advances in Mathematics}{Adv. in Math.}}
   \ITEE{#2}{ACS}{\ITE{\equal{#1}{+}}
      {Applied Categorical Structures}{Appl. Categor. Struct.}}
   \ITEE{#2}{ActaSM}{\ITE{\equal{#1}{+}}
      {Acta Scientiarum Mathematicarum}{Acta Sci. Math.}}
   \ITEE{#2}{AmJM}{\ITE{\equal{#1}{+}}
      {American Journal of Mathematics}{Amer. J. Math.}}
   \ITEE{#2}{AmMMon}{\ITE{\equal{#1}{+}}
      {American Mathematical Monthly}{Amer. Math. Mon.}}
   \ITEE{#2}{AnnSciEcNormSupT}{\ITE{\equal{#1}{+}}
      {Annales Scientifiques de l'\'{E}cole Normale Sup\'{e}rieure (3)}{Ann. Sci. \'{E}c. Norm. Sup\'{e}r. (3)}}
   \ITEE{#2}{AnnM}{\ITE{\equal{#1}{+}}
      {Annals of Mathematics}{Ann. Math.}}
   \ITEE{#2}{AnnProb}{\ITE{\equal{#1}{+}}
      {The Annals of Probability}{Ann. Probab.}}
   \ITEE{#2}{AnnPALog}{\ITE{\equal{#1}{+}}
      {Annals of Pure and Applied Logic}{Ann. Pure Appl. Logic}}
   \ITEE{#2}{APM}{\ITE{\equal{#1}{+}}
      {Annales Polonici Mathematici}{Ann. Polon. Math.}}
   \ITEE{#2}{ArchM}{\ITE{\equal{#1}{+}}
      {Archiv der Mathematik}{Arch. Math.}}
   \ITEE{#2}{AttiAccLincRendNat}{\ITE{\equal{#1}{+}}
      {Atti della Accademia Nazionale dei Lincei. Rendiconti. Classe di Scienze Fisiche, Matematiche e Naturali}
      {Atti Accad. Naz. Lincei Rend. Cl. Sci. Fis. Mat. Nat.}}
   \ITEE{#2}{BAMS}{\ITE{\equal{#1}{+}}
      {Bulletin of the American Mathematical Society}{Bull. Amer. Math. Soc.}}
   \ITEE{#2}{BAustrMS}{\ITE{\equal{#1}{+}}
      {Bulletin of the Australian Mathematical Society}{Bull. Austral. Math. Soc.}}
   \ITEE{#2}{BLondMS}{\ITE{\equal{#1}{+}}
      {Bulletin of the London Mathematical Sociecy}{Bull. Lond. Math. Soc.}}
   \ITEE{#2}{BAPolSSSM}{\ITE{\equal{#1}{+}}
      {Bulletin de l'Acad\'{e}mie Polonaise des Sciences. S\'{e}rie des Sciences Math\'{e}matiques}
      {Bull. Acad. Pol. Sci. S\'{e}r. Sci. Math.}}
   \ITEE{#2}{BullSM}{\ITE{\equal{#1}{+}}
      {Bulletin des Sciences Math\'{e}matiques}{Bull. Sci. Math.}}
   \ITEE{#2}{BullPol}{\ITE{\equal{#1}{+}}
      {Bulletin of the Polish Academy of Sciences: Mathematics}{Bull. Pol. Acad. Sci. Math.}}
   \ITEE{#2}{CanadJM}{\ITE{\equal{#1}{+}}
      {Canadian Journal Mathematics}{Canad. J. Math.}}
   \ITEE{#2}{CollectM}{\ITE{\equal{#1}{+}}
      {Collectanea Mathematica}{Collect. Math.}}
   \ITEE{#2}{CMUC}{\ITE{\equal{#1}{+}}
      {Commentationes Mathematicae Universitatis Carolinae}{Comment. Math. Univ. Carolin.}}
   \ITEE{#2}{CRParis}{\ITE{\equal{#1}{+}}
      {C. R. Paris}{C. R. Paris}}
   \ITEE{#2}{CRASParis}{\ITE{\equal{#1}{+}}
      {Comptes Rendus de l'Acad\'{e}mie des Sciences. Paris}{C. R. Acad. Sci. Paris}}
   \ITEE{#2}{CEurJM}{\ITE{\equal{#1}{+}}
      {Central European Journal of Mathematics}{Cent. Eur. J. Math.}}
   \ITEE{#2}{CMHelv}{\ITE{\equal{#1}{+}}
      {Commentarii Mathematici Helvetici}{Comment. Math. Helv.}}
   \ITEE{#2}{CollM}{\ITE{\equal{#1}{+}}
      {Colloquium Mathematicum}{Coll. Math.}}
   \ITEE{#2}{ComposM}{\ITE{\equal{#1}{+}}
      {Compositio Mathematica}{Compos. Math.}}
   \ITEE{#2}{CzMJ}{\ITE{\equal{#1}{+}}
      {Czechoslovak Mathematical Journal}{Czech. Math. J.}}
   \ITEE{#2}{DissM}{\ITE{\equal{#1}{+}}
      {Dissertationes Mathematicae}{Dissert. Math.}}
   \ITEE{#2}{DANSSSR}{\ITE{\equal{#1}{+}}
      {Doklady Akademii Nauk SSSR}{Dokl. Akad. Nauk SSSR}}
   \ITEE{#2}{DukeMJ}{\ITE{\equal{#1}{+}}
      {Duke Mathematical Journal}{Duke Math. J.}}
   \ITEE{#2}{ELA}{\ITE{\equal{#1}{+}}
      {The Electronic Journal of Linear Algebra}{Electron. J. Linear Algebra}}
   \ITEE{#2}{ExtrM}{\ITE{\equal{#1}{+}}
      {Extracta Mathematicae}{Extracta Math.}}
   \ITEE{#2}{FM}{\ITE{\equal{#1}{+}}
      {Fundamenta Mathematicae}{Fund. Math.}}
   \ITEE{#2}{FAA}{\ITE{\equal{#1}{+}}
      {Functional Analysis and its Applications}{Funct. Anal. Appl.}}
   \ITEE{#2}{FunkAnalPril}{\ITE{\equal{#1}{+}}
      {Funktsional'ny\u{\i} Analiz i Ego Prilozheniya}{Funkts. Anal. Prilozh.}}
   \ITEE{#2}{GTopA}{\ITE{\equal{#1}{+}}
      {General Topology and its Applications}{General Topol. Appl.}}
   \ITEE{#2}{HJM}{\ITE{\equal{#1}{+}}
      {Houston Journal of Mathematics}{Houston J. Math.}}
   \ITEE{#2}{IllinoisJM}{\ITE{\equal{#1}{+}}
      {Illinois Journal of Mathematics}{Illinois J. Math.}}
   \ITEE{#2}{IndagMP}{\ITE{\equal{#1}{+}}
      {Indagationes Mathematicae (Proceedings)}{Indagationes Math. Proc.}}
   \ITEE{#2}{IndianaUMJ}{\ITE{\equal{#1}{+}}
      {Indiana University Mathematical Journal}{Indiana Univ. Math. J.}}
   \ITEE{#2}{InHauEtSPM}{\ITE{\equal{#1}{+}}
      {Inst. Hautes \'{E}tudes Sci. Publ. Math.}{Inst. Hautes \'{E}tudes Sci. Publ. Math.}}
   \ITEE{#2}{IEOT}{\ITE{\equal{#1}{+}}
      {Integral Equations and Operator Theory}{Integr. Equ. Oper. Theory}}
   \ITEE{#2}{IsraelJM}{\ITE{\equal{#1}{+}}
      {Israel Journal of Mathematics}{Israel J. Math.}}
   \ITEE{#2}{JAusMSA}{\ITE{\equal{#1}{+}}
      {Journal of the Australian Mathematical Society. Series A}{J. Aust. Math. Soc. Ser. A}}
   \ITEE{#2}{JCA}{\ITE{\equal{#1}{+}}
      {Journal of Convex Analysis}{J. Convex Anal.}}
   \ITEE{#2}{JChinUST}{\ITE{\equal{#1}{+}}
      {J. China Univ. Sci. Tech.}{J. China Univ. Sci. Tech.}}
   \ITEE{#2}{JFA}{\ITE{\equal{#1}{+}}
      {Journal of Functional Analysis}{J. Funct. Anal.}}
   \ITEE{#2}{JKoreanMS}{\ITE{\equal{#1}{+}}
      {Journal of the Korean Mathematical Society}{J. Korean Math. Soc.}}
   \ITEE{#2}{JMAnApp}{\ITE{\equal{#1}{+}}
      {J. Math. Anal. Appl.}{J. Math. Anal. Appl.}}
   \ITEE{#2}{JOT}{\ITE{\equal{#1}{+}}
      {Journal of Operator Theory}{J. Operator Theory}}
   \ITEE{#2}{KodaiMSemRep}{\ITE{\equal{#1}{+}}
      {Kodai Math. Sem. Rep.}{Kodai Math. Sem. Rep.}}
   \ITEE{#2}{LAA}{\ITE{\equal{#1}{+}}
      {Linear Algebra and its Applications}{Linear Algebra Appl.}}
   \ITEE{#2}{LMLA}{\ITE{\equal{#1}{+}}
      {Linear and Multilinear Algebra}{Linear Multilinear Algebra}}
   \ITEE{#2}{LNM}{\ITE{\equal{#1}{+}}
      {Lecture Notes in Mathematics}{Lecture Notes Math.}}
   \ITEE{#2}{MathJap}{\ITE{\equal{#1}{+}}
      {Math. Japon.}{Math. Japon.}}
   \ITEE{#2}{MLQ}{\ITE{\equal{#1}{+}}
      {Mathematical Logic Quarterly}{Math. Log. Q.}}
   \ITEE{#2}{MProcCambPhS}{\ITE{\equal{#1}{+}}
      {Mathematical Proceedings of the Cambridge Philosophical Society}{Math. Proc. Cambridge Phil. Soc.}}
   \ITEE{#2}{MMag}{\ITE{\equal{#1}{+}}
      {Mathematics Magazine}{Math. Mag.}}
   \ITEE{#2}{MSb}{\ITE{\equal{#1}{+}}
      {Matematicheski\u{\i} Sbornik}{Mat. Sb.}}
   \ITEE{#2}{MStud}{\ITE{\equal{#1}{+}}
      {Matematychni Studi\"{\i}}{Mat. Stud.}}
   \ITEE{#2}{MScand}{\ITE{\equal{#1}{+}}
      {Mathematica Scandinavica}{Math. Scand.}}
   \ITEE{#2}{MAnn}{\ITE{\equal{#1}{+}}
      {Mathematische Annalen}{Math. Ann.}}
   \ITEE{#2}{MAMS}{\ITE{\equal{#1}{+}}
      {Memoirs of the American Mathematical Society}{Mem. Amer. Math. Soc.}}
   \ITEE{#2}{MichMJ}{\ITE{\equal{#1}{+}}
      {Michigan Mathematical Journal}{Mich. Math. J.}}
   \ITEE{#2}{MonatM}{\ITE{\equal{#1}{+}}
      {Monatshefte f\"{u}r Mathematik}{Mh. Math.}}
   \ITEE{#2}{NonlinA}{\ITE{\equal{#1}{+}}
      {Nonlinear Analysis: Theory, Methods \& Applications}{Nonlinear Anal.}}
   \ITEE{#2}{NAMS}{\ITE{\equal{#1}{+}}
      {Notices of the American Mathematical Society}{Notices Amer. Math. Soc.}}
   \ITEE{#2}{OpusM}{\ITE{\equal{#1}{+}}
      {Opuscula Mathematica}{Opuscula Math.}}
   \ITEE{#2}{PacJM}{\ITE{\equal{#1}{+}}
      {Pacific Journal of Mathematics}{Pacific J. Math.}}
   \ITEE{#2}{PeriodMHung}{\ITE{\equal{#1}{+}}
      {Periodica Mathematica Hungarica}{Period. Math. Hungarica}}
   \ITEE{#2}{PAMS}{\ITE{\equal{#1}{+}}
      {Proceedings of the American Mathematical Society}{Proc. Amer. Math. Soc.}}
   \ITEE{#2}{ProcCambPhS}{\ITE{\equal{#1}{+}}
      {Proceedings of the Cambridge Philosophical Society}{Proc. Cambridge Phil. Soc.}}
   \ITEE{#2}{ProcImpAcadTokyo}{\ITE{\equal{#1}{+}}
      {Proc. Imp. Acad. Tokyo}{Proc. Imp. Acad. Tokyo}}
   \ITEE{#2}{ProcKonink}{\ITE{\equal{#1}{+}}
      {Proceedings of the Koninklijke Nederlandse Akademie van Wetenschappen}{Nederl. Akad. Wetensch. Proc. Ser. A}}
   \ITEE{#2}{PLondMS}{\ITE{\equal{#1}{+}}
      {Proceedings of the London Mathematical Society}{Proc. London Math. Soc.}}
   \ITEE{#2}{PNatlUSA}{\ITE{\equal{#1}{+}}
      {Proceedings of the National Academy of Sciences of the United States of America}{Proc. Natl. Acad. Sci. USA}}
   \ITEE{#2}{PublRIMSKyoto}{\ITE{\equal{#1}{+}}
      {Publ. Res. Inst. Math. Sci. Kyoto Univ.}{Publ. Res. Inst. Math. Sci.}}
   \ITEE{#2}{PWN}{\ITE{\equal{#1}{+}}
      {PWN -- Polish Scientific Publishers, Warszawa}{PWN -- Polish Scientific Publishers, Warszawa}}
   \ITEE{#2}{RCMP}{\ITE{\equal{#1}{+}}
      {Rendiconti del Circolo Matematico di Palermo}{Rend. Circ. Mat. Palermo}}
   \ITEE{#2}{RussMS}{\ITE{\equal{#1}{+}}
      {Russian Mathematical Surveys}{Russian Math. Surveys}}
   \ITEE{#2}{SbM}{\ITE{\equal{#1}{+}}
      {Sbornik: Mathematics}{Sb. Math.}}
   \ITEE{#2}{SciRepTokyoA}{\ITE{\equal{#1}{+}}
      {Science Reports of Tokyo Kyoiku Daigaku, Section A}{Sci. Rep. Tokyo Kyoiku Daigaku Sect. A}}
   \ITEE{#2}{SeminProbStras}{\ITE{\equal{#1}{+}}
      {S\'{e}minaire de probabilit\'{e}s de Strasbourg}{S\'{e}min. Prob. Strasbourg}}
   \ITEE{#2}{SIAMJMAA}{\ITE{\equal{#1}{+}}
      {SIAM Journal on Matrix Analysis and Applications}{SIAM J. Matrix Anal. Appl.}}
   \ITEE{#2}{SibirMZ}{\ITE{\equal{#1}{+}}
      {Sibirski\v{\i} Mat. \v{Z}hurnal}{Sibirsk. Mat. \v{Z}.}}
   \ITEE{#2}{SM}{\ITE{\equal{#1}{+}}
      {Studia Mathematica}{Studia Math.}}
   \ITEE{#2}{TAMS}{\ITE{\equal{#1}{+}}
      {Transactions of the American Mathematical Society}{Trans. Amer. Math. Soc.}}
   \ITEE{#2}{TohokuMJ}{\ITE{\equal{#1}{+}}
      {T\^{o}hoku Mathematical Journal}{T\^{o}hoku Math. J.}}
   \ITEE{#2}{TomskUnivRev}{\ITE{\equal{#1}{+}}
      {Tomsk Universitet Review}{Tomsk. Univ. Rev.}}
   \ITEE{#2}{TopA}{\ITE{\equal{#1}{+}}
      {Topology and its Applications}{Topology Appl.}}
   \ITEE{#2}{TopMethNA}{\ITE{\equal{#1}{+}}
      {Topological Methods in Nonlinear Analysis}{Topol. Methods Nonlinear Anal.}}
   \ITEE{#2}{TsukubaJM}{\ITE{\equal{#1}{+}}
      {Tsukuba Journal of Mathematics}{Tsukuba J. Math.}}
   \ITEE{#2}{UspekhiMN}{\ITE{\equal{#1}{+}}
      {Uspekhi Matem. Nauk}{Uspekhi Mat. Nauk}}
   }

\newcommand{\paplist}[3][]{
   \ITEE{#3}{NIAkhiezer,IMGlazman1993}{
      \BIb{#2}{N.I. Akhiezer and I.M. Glazman}
         {Theory of Linear Operators in Hilbert Space}
         {Dover Publications, Inc., New York}{1993}{#1}}
   \ITEE{#3}{RDAnderson1967}{
      \BIB{#2}{R.D. Anderson}
         {On topological infinite deficiency}
         {\jRN{MichMJ}}{14}{1967}{365--383}{#1}}
   \ITEE{#3}{RDAnderson,JMcCharen1970}{
      \BIB{#2}{R.D. Anderson and J. McCharen}
         {On extending homeomorphisms to Fr\'{e}chet manifolds}
         {\jRN{PAMS}}{25}{1970}{283--289}{#1}}
   \ITEE{#3}{RDAnderson,DWCurtis,JVanMill1982}{
      \BIB{#2}{R.D. Anderson, D.W. Curtis, J. van Mill}
         {A fake topological Hilbert space}
         {\jRN{TAMS}}{272}{1982}{311--321}{#1}}
   \ITEE{#3}{RArens,JEells1956}{
      \BIB{#2}{R. Arens and J. Eells}
         {On embedding uniform and topological spaces}
         {\jRN{PacJM}}{6}{1956}{397--403}{#1}}
   \ITEE{#3}{NAronszajn,PPanitchpakdi1956}{
      \BIB{#2}{N. Aronszajn and P. Panitchpakdi}
         {Extension of uniformly continuous transformations and hyperconvex metric spaces}
         {\jRN{PacJM}}{6}{1956}{405--439}{#1}}
   \ITEE{#3}{KJBabenko1948}{
      \BIB{#2}{K.J. Babenko}
         {On conjugate functions}
         {\jRN{DANSSSR}}{62}{1948}{157--160}{#1}}
   \ITEE{#3}{TBanakh1995}{
      \BIB{#2}{T.O. Banakh}
         {Topology of spaces of probability measures, I}
         {\jRN{MStud}}{5}{1995}{65--87 (Russian)}{#1}}
   \ITEE{#3}{TBanakh1995a}{
      \BIB{#2}{T.O. Banakh}
         {Topology of spaces of probability measures, II}
         {\jRN{MStud}}{5}{1995}{88--106 (Russian)}{#1}}
   \ITEE{#3}{TBanakh1998}{
      \BIB{#2}{T. Banakh}
         {Characterization of spaces admitting a homotopy dense embedding into a Hilbert manifold}
         {\jRN{TopA}}{86}{1998}{123--131}{#1}}
   \ITEE{#3}{TBanakh,TNRadul1997}{
      \BIB{#2}{T.O. Banakh and T.N. Radul}
         {Topology of spaces of probability measures}
         {\jRN{SbM}}{188}{1997}{973--995}{#1}}
   \ITEE{#3}{TBanakh,TRadul,MZarichnyi1996}{
      \BIb{#2}{T. Banakh, T. Radul, M. Zarichnyi}
         {Absorbing sets in infinite-dimensional manifolds}
         {VNTL Publishers, Lviv}{1996}{#1}}
   \ITEE{#3}{TBanakh,IZarichnyy2008}{
      \BIB{#2}{T. Banakh and I. Zarichnyy}
         {Topological groups and convex sets homeomorphic to non-separable Hilbert spaces}
         {\jRN{CEurJM}}{6}{2008}{77--86}{#1}}
   \ITEE{#3}{HBecker,ASKechris1996}{
      \BIb{#2}{H. Becker and A.S. Kechris}
         {The Descriptive Set Theory of Polish Group Actions \textup{(London Math. Soc. Lecture Note Series, vol. 232)}}
         {University Press, Cambridge}{1996}{#1}}
   \ITEE{#3}{GBeer1993}{
      \BIb{#2}{G. Beer}
         {Topologies on Closed and Closed Convex Sets \textup{(Mathematics and Its Applications)}}
         {Kluwer Academic Publishers, Dordrecht}{1993}{#1}}
   \ITEE{#3}{NEBenamara,NNikolski1999}{
      \BIB{#2}{N.E. Benamara and N. Nikolski}
         {Resolvent tests for similarity to a normal operator}
         {\jRN{PLondMS}}{78}{1999}{585--626}{#1}}
   \ITEE{#3}{YBenyamini,JLindenstrauss2000}{
      \BIb{#2}{Y. Benyamini and J. Lindenstrauss}
         {Geometric nonlinear functional analysis I}
         {AMS Colloquium Publications 48}{2000}{#1}}
   \ITEE{#3}{SKBerberian1974}{
      \BIb{#2}{S.K. Berberian}
         {Lectures in Functional Analysis and Operator Theory}
         {Graduate Texts in Mathematics 15, Springer-Verlag, New York}{1974}{#1}}
   \ITEE{#3}{SNBernstein1954}{
      \BIb{#2}{S.N. Bernstein}
         {Collected Works II}
         {Akad. Nauk SSSR, Moscow}{1954 (Russian)}{#1}}
   \ITEE{#3}{CzBessaga,APelczynski1972}{
      \BIB{#2}{Cz. Bessaga and A. Pe\l{}czy\'{n}ski}
         {On spaces of measurable functions}
         {\jRN{SM}}{44}{1972}{597--615}{#1}}
   \ITEE{#3}{CzBessaga,APelczynski1975}{
      \BIb{#2}{Cz. Bessaga and A. Pe\l{}czy\'{n}ski}
         {Selected topics in infinite-dimensional topology}
         {\jRN{PWN}}{1975}{#1}}
   \ITEE{#3}{MBestvina,JMogilski1986}{
      \BIB{#2}{M. Bestvina and J. Mogilski}
         {Characterizing certain incomplete infinite-dimensional absolute retracts}
         {\jRN{MichMJ}}{33}{1986}{291--313}{#1}}
   \ITEE{#3}{MBestvina,PBowers,JMogilsky,JWalsh1986}{
      \BIB{#2}{M. Bestvina, P. Bowers, J. Mogilsky, J. Walsh}
         {Characterization of Hilbert space manifolds revisited}
         {\jRN{TopA}}{24}{1986}{53--69}{#1}}
   \ITEE{#3}{RBhatia1997}{
      \BIb{#2}{R. Bhatia}
         {Matrix Analysis}
         {Springer, New York}{1997}{#1}}
   \ITEE{#3}{GBirkhoff1936}{
      \BIB{#2}{G. Birkhoff}
         {A note on topological groups}
         {\jRN{ComposM}}{3}{1936}{427--430}{#1}}
   \ITEE{#3}{MSBirman,MZSolomjak1987}{
      \BIb{#2}{M.S. Birman and M.Z. Solomjak}
         {Spectral Theory of Self-Adjoint Operators in Hilbert Space}
         {D. Reidel Publishing Co., Dordrecht}{1987}{#1}}
   \ITEE{#3}{EBishop1961}{
      \BIB{#2}{E. Bishop}
         {A generalization of the Stone-Weierstrass theorem}
         {\jRN{PacJM}}{11}{1961}{777--783}{#1}}
   \ITEE{#3}{BBlackadar2006}{
      \BIb{#2}{B. Blackadar}{Operator Algebras. Theory of $\CCc^*$-algebras and von Neumann algebras 
         \textup{(Encyclopaedia of Mathematical Sciences, vol. 122: Operator Algebras and Non-Commutative Geometry III)}}
         {Springer-Verlag, Berlin-Heidelberg}{2006}{#1}}
   \ITEE{#3}{JBlass,WHolsztynski1972}{
      \BIB{#2}{J. Blass and W. Holszty\'{n}ski}
         {Cubical polyhedra and homotopy III}
         {\jRN{AttiAccLincRendNat}}{53}{1972}{275--279}{#1}}
   \ITEE{#3}{FFBonsall,NJDuncan1973}{
      \BIb{#2}{F.F. Bonsall and N.J. Duncan}
         {Complete Normed Algebras}
         {Springer Verlag, Berlin}{1973}{#1}}
   \ITEE{#3}{NBourbaki2002}{
      \BIb{#2}{N. Bourbaki}
         {Lie Groups and Lie Algebras, Chapters 4--6}
         {Springer, New York}{2002}{#1}}
   \ITEE{#3}{PLBowers1989}{
      \BIB{#2}{P.L. Bowers}
         {Limitation topologies on function spaces}
         {\jRN{TAMS}}{314}{1989}{421--431}{#1}}
   \ITEE{#3}{ABrown1953}{
      \BIB{#2}{A. Brown}
         {On a class of operators}
         {\jRN{PAMS}}{4}{1953}{723--728}{#1}}
   \ITEE{#3}{ABrown,CKFong,DWHadwin1978}{
      \BIB{#2}{A. Brown, C.-K. Fong, D.W. Hadwin}
         {Parts of operators on Hilbert space}
         {\jRN{IllinoisJM}}{22}{1978}{306--314}{#1}}
   \ITEE{#3}{AMBruckner,JBBruckner,BSThomson1997}{
      \BIb{#2}{A.M. Bruckner, J.B. Bruckner, B.S. Thomson}
         {Real Analysis}
         {Prentice-Hall, New Jersey}{1997}{#1}}
   \ITEE{#3}{PJCameron,AMVershik2006}{
      \BIB{#2}{P.J. Cameron and A.M. Vershik}
         {Some isometry groups of Urysohn space}
         {\jRN{AnnPALog}}{143}{2006}{70--78}{#1}}
   \ITEE{#3}{CCastaing1966}{
      \BIB{#2}{C. Castaing}
         {Quelques probl\`{e}mes de mesurabilit\'{e} li\'{e}es \`{a} la th\'{e}orie de la commande}
         {\jRN{CRParis}}{262}{1966}{409--411}{#1}}
   \ITEE{#3}{JAVanCasteren1980}{
      \BIB{#2}{J.A. van Casteren}
         {A problem of Sz.-Nagy}
         {\jRN{ActaSM}}{42}{1980}{189--194}{#1}}
   \ITEE{#3}{JAVanCasteren1983}{
      \BIB{#2}{J.A. van Casteren}
         {Operators similar to unitary or selfadjoint ones}
         {\jRN{PacJM}}{104}{1983}{241--255}{#1}}
   \ITEE{#3}{XCatepillan,MPtak,WSzymanski1994}{
      \BIB{#2}{X. Catepill\'{a}n, M. Ptak, W. Szyma\'{n}ski}
         {Multiple canonical decompositions of families of operators and a model of quasinormal families}
         {\jRN{PAMS}}{121}{1994}{1165--1172}{#1}}
   \ITEE{#3}{RCauty1994}{
      \BIB{#2}{R. Cauty}
         {Un espace m\'{e}trique lin\'{e}aire qui n'est pas un r\'{e}tracte absolu}
         {\jRN{FM}}{146}{1994}{85--99, (French)}{#1}}
   \ITEE{#3}{TAChapman1971}{
      \BIB{#2}{T.A. Chapman}
         {Deficiency in infinite-dimensional manifolds}
         {\jRN{GTopA}}{1}{1971}{263--272}{#1}}
   \ITEE{#3}{TAChapman1976}{
      \BIb{#2}{T.A. Chapman}
         {Lectures on Hilbert cube manifolds}
         {C.B.M.S. Regional Conference Series in Math. No 28, Amer. Math. Soc.}{1976}{#1}}
   \ITEE{#3}{RBChuaqui1977}{
      \BIB{#2}{R.B. Chuaqui}
         {Measures invariant under a group of transformations}
         {\jRN{PacJM}}{68}{1977}{313--329}{#1}}
   \ITEE{#3}{JBConway1985}{
      \BIb{#2}{J.B. Conway}
         {A Course in Functional Analysis}
         {Springer-Verlag, New York}{1985}{#1}}
   \ITEE{#3}{JBConway2000}{
      \BIb{#2}{J.B. Conway}
         {A Course in Operator Theory}
         {(Graduate Studies in Mathematics, vol. 21) Amer. Math. Soc., Providence}{2000}{#1}}
   \ITEE{#3}{GCorach,AMaestripieri,MMbekhta2009}{
      \BIB{#2}{G. Corach, A. Maestripieri, M. Mbekhta}
         {Metric and homogeneous structure of closed range operators}
         {\jRN{JOT}}{61}{2009}{171--190}{#1}}
   \ITEE{#3}{MJCowen,RGDouglas1978}{
      \BIB{#2}{M.J. Cowen and R.G. Douglas}
         {Complex geometry and operator theory}
         {\jRN{ActaM}}{141}{1978}{187--261}{#1}}
   \ITEE{#3}{DWCurtis1985}{
      \BIB{#2}{D.W. Curtis}
         {Boundary sets in the Hilbert cube}
         {\jRN{TopA}}{20}{1985}{201--221}{#1}}
   \ITEE{#3}{MMDay1958}{
      \BIb{#2}{M.M. Day}
         {Normed Linear Spaces}
         {Springer Verlag, Berlin}{1958}{#1}}
   \ITEE{#3}{CDellacherie1967}{
      \BIB{#2}{C. Dellacherie}
         {Un compl\'{e}ment au th\'{e}or\`{e}me de Weierstrass-Stone}
         {\jRN{SeminProbStras}}{1}{1967}{52--53}{#1}}
   \ITEE{#3}{JJDijkstra1987}{
      \BIB{#2}{J.J. Dijkstra}
         {Strong negligibility of $\sigma$-compacta does not characterize Hilbert space}
         {\jRN{PacJM}}{127}{1987}{19--30}{#1}}
   \ITEE{#3}{JJDijkstra1990}{
      \BIB{#2}{J.J. Dijkstra}
         {Characterizing Hilbert space topology in terms of strong negligibility}
         {\jRN{ComposM}}{75}{1990}{299--306}{#1}}
   \ITEE{#3}{TDobrowolski,WMarciszewski2002}{
      \BIB{#2}{T. Dobrowolski and W. Marciszewski}
         {Failure of the Factor Theorem for Borel pre-Hilbert spaces}
         {\jRN{FM}}{175}{2002}{53--68}{#1}}
   \ITEE{#3}{TDobrowolski,JMogilski1990}{
      \BiB{#2}{T. Dobrowolski and J. Mogilski}
         {Problems on Topological Classification of Incomplete Metric Spaces}{Chapter 25 in:}
         {Open Problems in Topology}{J. van Mill and G.M. Reed (eds.), North-Holland Amsterdam}{1990}{411--429}{#1}}
   \ITEE{#3}{TDobrowolski,HTorunczyk1981}{
      \BIB{#2}{T. Dobrowolski and H. Toru\'{n}czyk}
         {Separable complete ANR's admitting a group structure are Hilbert manifolds}
         {\jRN{TopA}}{12}{1981}{229--235}{#1}}
   \ITEE{#3}{RGDouglas1966}{
      \BIB{#2}{R.G. Douglas}
         {On majorization, factorization and range inclusion of operators in Hilbert space}
         {\jRN{PAMS}}{17}{1966}{413--416}{#1}}
   \ITEE{#3}{CHDowker1947}{
      \BIB{#2}{C.H. Dowker}
         {Mapping theorems for non-compact spaces}
         {\jRN{AmJM}}{69}{1947}{200--242}{#1}}
   \ITEE{#3}{CHDowker1952}{
      \BIB{#2}{C.H. Dowker}
         {Topology of metric complexes}
         {\jRN{AmJM}}{74}{1952}{555--577}{#1}}
   \ITEE{#3}{JDugundji1951}{
      \BIB{#2}{J. Dugundji}
         {An extension of Tietze's theorem}
         {\jRN{PacJM}}{1}{1951}{353--367}{#1}}
   \ITEE{#3}{JDugundji1958}{
      \BIB{#2}{J. Dugundji}
         {Absolute neighborhood retracts and local connectedness for arbitrary metric spaces}
         {\jRN{ComposM}}{13}{1958}{229--246}{#1}}
   \ITEE{#3}{JDugundji1965}{
      \BIB{#2}{J. Dugundji}
         {Locally equiconnected spaces and absolute neighborhood retracts}
         {\jRN{FM}}{57}{1965}{187--193}{#1}}
   \ITEE{#3}{NDunford,JTSchwartz1958}{
      \BIb{#2}{N. Dunford and J.T. Schwartz}
         {Linear Operators, part I}
         {Interscience Publishers, New York}{1958}{#1}}
   \ITEE{#3}{NDunford,JTSchwartz1963}{
      \BIb{#2}{N. Dunford and J.T. Schwartz}
         {Linear Operators, part II}
         {Interscience Publishers, New York}{1963}{#1}}
   \ITEE{#3}{NDunford,JTSchwartz1971}{
      \BIb{#2}{N. Dunford and J.T. Schwartz}
         {Linear Operators, part III}
         {Wiley-Interscience, New York}{1971}{#1}}
   \ITEE{#3}{MLEaton,MDPerlman1977}{
      \BIB{#2}{M.L. Eaton and M.D. Perlman}
         {Reflection groups, generalized Schur functions and the geometry of majorization}
         {\jRN{AnnProb}}{5}{1977}{829--860}{#1}}
   \ITEE{#3}{EGEffros1965}{
      \BIB{#2}{E.G. Effros}
         {The Borel space of von Neumann algebras on a separable Hilbert space}
         {\jRN{PacJM}}{15}{1965}{1153--1164}{#1}}
   \ITEE{#3}{EGEffros1966}{
      \BIB{#2}{E.G. Effros}
         {Global structure in von Neumann algebras}
         {\jRN{TAMS}}{121}{1966}{434--454}{#1}}
   \ITEE{#3}{REspinola,MAKhamsi2001}{
      \BiB{#2}{R. Espinola and M.A. Khamsi}
         {Introduction to hyperconvex spaces}{Chapter XIII in:}{Handbook of Metric Fixed Point Theory}
         {W.A. Kirk and B. Sims (editors), Kluwer Academic Publishers}{2001}{391--435}{#1}}
   \ITEE{#3}{PAFillmore,JPWilliams1971}{
      \BIB{#2}{P.A. Fillmore and J.P. Williams}
         {On operator ranges}
         {\jRN{AdvM}}{7}{1971}{254--281}{#1}}
   \ITEE{#3}{JEells,NHKuiper1969}{
      \BIB{#2}{J. Eells and N.H. Kuiper}
         {Homotopy negligible subsets in infinite-dimensional manifolds}
         {\jRN{ComposM}}{21}{1969}{151--161}{#1}}
   \ITEE{#3}{REngelking1977}{
      \BIb{#2}{R. Engelking}
         {General Topology}
         {\jRN{PWN}}{1977}{#1}}
   \ITEE{#3}{REngelking1978}{
      \BIb{#2}{R. Engelking}
         {Dimension Theory}
         {\jRN{PWN}}{1978}{#1}}
   \ITEE{#3}{REngelking1989}{
      \BIb{#2}{R. Engelking}
         {General Topology. Revised and completed edition \textup{(Sigma series in pure mathematics, vol. 6)}}
         {Heldermann Verlag, Berlin}{1989}{#1}}
   \ITEE{#3}{PErdos,RDMauldin1976}{
      \BIB{#2}{P. Erd\"{o}s and R.D. Mauldin}
         {The nonexistence of certain invariant measures}
         {\jRN{PAMS}}{59}{1976}{321--322}{#1}}
   \ITEE{#3}{JErnest1976}{
      \BIB{#2}{J. Ernest}
         {Charting the operator terrain}
         {\jRN{MAMS}}{171}{1976}{207 pp}{#1}}
   \ITEE{#3}{RHFox1943}{
      \BIB{#2}{R.H. Fox}
         {On fiber spaces, II}
         {\jRN{BAMS}}{49}{1943}{733--735}{#1}}
   \ITEE{#3}{NAFriedman1970}{
      \BIb{#2}{N.A. Friedman}
         {Introduction to ergodic theory}
         {Van Nostrand Reinhold Company}{1970}{#1}}
   \ITEE{#3}{MFujii,MKajiwara,YKato,FKubo1976}{
      \BIB{#2}{M. Fujii, M. Kajiwara, Y. Kato, F. Kubo}
         {Decompositions of operators in Hilbert spaces}
         {\jRN{MathJap}}{21}{1976}{117--120}{#1}}
   \ITEE{#3}{SGao,ASKechris2003}{
      \BIB{#2}{S. Gao and A.S. Kechris}
         {On the classification of Polish metric spaces up to isometry}
         {\jRN{MAMS}}{161}{2003}{viii+78}{#1}}
   \ITEE{#3}{MIGarrido,FMontalvo1991}{
      \BIB{#2}{M.I. Garrido and F. Montalvo}
         {On some generalizations of the Kakutani-Stone and Stone-Weierstrass theorems}
         {\jRN{ExtrM}}{6}{1991}{156--159}{#1}}
   \ITEE{#3}{LGe,JShen2002}{
      \BIB{#2}{L. Ge and J. Shen}
         {Generator problem for certain property T factors}
         {\jRN{PNAS}}{99}{2002}{565--567}{#1}}
   \ITEE{#3}{IMGelfand,MANaimark1943}{
      \BIB{#2}{I.M. Gelfand and M.A. Naimark}
         {On the embedding of normed rings into the ring of operators in Hilbert space}
         {\jRN{MSb}}{12}{1943}{197--213}{#1}}
   \ITEE{#3}{FGesztesy,MMalamud,MMitrea,SNaboko2009}{
      \BIB{#2}{F. Gesztesy, M. Malamud, M. Mitrea, S. Naboko}
         {Generalized polar decompositions for closed operators in Hilbert spaces and some applications}
         {\jRN{IEOT}}{64}{2009}{83--113}{#1}}
   \ITEE{#3}{LGillman,MJerison1960}{
      \BIb{#2}{L. Gillman and M. Jerison}
         {Rings of continuous functions}
         {New York}{1960}{#1}}
   \ITEE{#3}{JGlimm1960}{
      \BIB{#2}{J. Glimm}
         {A Stone-Weierstrass theorem for $\CCc^*$-algebras}
         {\jRN{AnnM}}{72}{1960}{216--244}{#1}}
   \ITEE{#3}{GGodefroy,NJKalton2003}{
      \BIB{#2}{G. Godefroy and N.J. Kalton}
         {Lipschitz-free Banach spaces}
         {\jRN{SM}}{159}{2003}{121--141}{#1}}
   \ITEE{#3}{ICGohberg,MGKrein1967}{
      \BIB{#2}{I.C. Gohberg and M.G. Krein}
         {On a description of contraction operators similar to unitary ones}
         {\jRN{FunkAnalPril}}{1}{1967}{38--60}{#1}}
   \ITEE{#3}{ELGriffinJr1953}{
      \BIB{#2}{E.L. Griffin Jr.}
         {Some contributions to the theory of rings of operators}
         {\jRN{TAMS}}{75}{1953}{471--504}{#1}}
   \ITEE{#3}{ELGriffinJr1955}{
      \BIB{#2}{E.L. Griffin Jr.}
         {Some contributions to the theory of rings of operators II}
         {\jRN{TAMS}}{79}{1955}{389--400}{#1}}
   \ITEE{#3}{MGromov1981}{
      \BIB{#2}{M. Gromov}
         {Groups of polynomial growth and expanding maps}
         {\jRN{InHauEtSPM}}{53}{1981}{53--73}{#1}}
   \ITEE{#3}{MGromov1999}{
      \BIb{#2}{M. Gromov}
         {Metric Structures for Riemannian and Non-Riemannian Spaces}
         {Progress in Math. \textbf{152}, Birkh\"{a}user}{1999}{#1}}
   \ITEE{#3}{JDeGroot1956}{
      \BIB{#2}{J. de Groot}
         {Non-archimedean metrics in topology}
         {\jRN{PAMS}}{7}{1956}{948--953}{#1}}
   \ITEE{#3}{LCGrove,CTBenson1985}{
      \BIb{#2}{L.C. Grove and C.T. Benson}
         {Finite Reflection Group}
         {2nd ed., Springer-Verlag}{1985}{#1}}
   \ITEE{#3}{VIGurarii1966}{
      \BIB{#2}{V.I. Gurari\v{\i}}{Spaces of universal placement, isotropic spaces and a problem of Mazur 
         on rotations of Banach spaces \textup{(Russian)}}
         {\jRN{SibirMZ}}{7}{1966}{1002--1013}{#1}}
   \ITEE{#3}{DWHadwin1976}{
      \BIB{#2}{D.W. Hadwin}
         {An operator-valued spectrum}
         {\jRN{NAMS}}{23}{1976}{A-163}{#1}}
   \ITEE{#3}{DWHadwin1977}{
      \BIB{#2}{D.W. Hadwin}
         {An operator-valued spectrum}
         {\jRN{IndianaUMJ}}{26}{1977}{329--340}{#1}}
   \ITEE{#3}{HHahn1932}{
      \BIb{#2}{H. Hahn}
         {Reelle Funktionen I}
         {Leipzig}{1932}{#1}}
   \ITEE{#3}{PRHalmos1950}{
      \BIb{#2}{P.R. Halmos}
         {Measure theory}
         {Van Nostrand, New York}{1950}{#1}}
   \ITEE{#3}{PRHalmos1951}{
      \BIb{#2}{P.R. Halmos}
         {Introduction to Hilbert Space and the Theory of Spectral Multiplicity}
         {Chelsea Publishing Company, New York}{1951}{#1}}
   \ITEE{#3}{PRHalmos1956}{
      \BIb{#2}{P.R. Halmos}
         {Lectures on Ergodic Theory}
         {Publ. Math. Soc. Japan, Tokyo}{1956}{#1}}
   \ITEE{#3}{PRHalmos1982}{
      \BIb{#2}{P.R. Halmos}
         {A Hilbert Space Problem Book}
         {Springer-Verlag New York Inc.}{1982}{#1}}
  \ITEE{#3}{PRHalmos,JEMcLaughlin1963}{
      \BIB{#2}{P.R. Halmos and J.E. McLaughlin}
         {Partial isometries}
         {\jRN{PacJM}}{13}{1963}{585--596}{#1}}
   \ITEE{#3}{RWHansell1972}{
      \BIB{#2}{R.W. Hansell}
         {On the nonseparable theory of Borel and Souslin sets}
         {\jRN{BAMS}}{78}{1972}{236--241}{#1}}
   \ITEE{#3}{FHausdorff1930}{
      \BIB{#2}{F. Hausdorff}
         {Erweiterung einer Hom\"{o}omorphie}
         {\jRN{FM}}{16}{1930}{353--360}{#1}}
   \ITEE{#3}{FHausdorff1934}{
      \BIB{#2}{F. Hausdorff}
         {\"{U}ber innere Abbildungen}
         {\jRN{FM}}{23}{1934}{279--291}{#1}}
   \ITEE{#3}{FHausdorff1938}{
      \BIB{#2}{F. Hausdorff}
         {Erweiterung einer stetigen Abbildung}
         {\jRN{FM}}{30}{1938}{40--47}{#1}}
   \ITEE{#3}{DWHenderson1971}{
      \BIB{#2}{D.W. Henderson}
         {Corrections and extensions of two papers about infinite-dimensional manifolds}
         {\jRN{GTopA}}{1}{1971}{321--327}{#1}}
   \ITEE{#3}{DWHenderson1975}{
      \BIB{#2}{D.W. Henderson}
         {$Z$-sets in ANR's}
         {\jRN{TAMS}}{213}{1975}{205--216}{#1}}
   \ITEE{#3}{DWHenderson,RMSchori1970}{
      \BIB{#2}{D.W. Henderson and R.M. Schori}
         {Topological classification of infinite-dimensional manifolds by homotopy type}
         {\jRN{BAMS}}{76}{1970}{121--124}{#1}}
   \ITEE{#3}{DWHenderson,JEWest1970}{
      \BIB{#2}{D.W. Henderson and J.E. West}
         {Triangulated infinite-dimensional manifolds}
         {\jRN{BAMS}}{76}{1970}{655--660}{#1}}
   \ITEE{#3}{BHoffmann1979}{
      \BIB{#2}{B. Hoffmann}
         {A compact contractible topological group is trivial}
         {\jRN{ArchM}}{32}{1979}{585--587}{#1}}
   \ITEE{#3}{DHofmann2002}{
      \BIB{#2}{D. Hofmann}
         {On a generalization of the Stone-Weierstrass theorem}
         {\jRN{ACS}}{10}{2002}{569--592}{#1}}
   \ITEE{#3}{GHognas,AMukherjea1995}{
      \BIb{#2}{G. H\"ogn\"as and A. Mukherjea}
         {Probability Measures on Semigroups. Convolution Products, Random Walks, and Random Matrices}
         {Plenum Press, New York}{1995}{#1}}
   \ITEE{#3}{MRHolmes1992}{
      \BIB{#2}{M.R. Holmes}
         {The universal separable metric space of Urysohn and isometric embeddings thereof in Banach spaces}
         {\jRN{FM}}{140}{1992}{199--223}{#1}}
   \ITEE{#3}{MRHolmes2008}{
      \BIB{#2}{M.R. Holmes}
         {The Urysohn space embeds in Banach spaces in just one way}
         {\jRN{TopA}}{155}{2008}{1479--1482}{#1}}
   \ITEE{#3}{RRHolmes,TYTam1999}{
      \BIB{#2}{R.R. Holmes and T.Y. Tam}
         {Distance to the convex hull of an orbit under the action of a compact group}
         {\jRN{JAusMSA}}{66}{1999}{331--357}{#1}}
   \ITEE{#3}{RHorn,RMathias1990}{
      \BIB{#2}{R. Horn and R. Mathias}
         {Cauchy-Schwartz inequalities associated with positive semidefinite matrices}
         {\jRN{LAA}}{142}{1990}{63--82}{#1}}
   \ITEE{#3}{GEHuhunaisvili1955}{
      \BIB{#2}{G.E. Huhunai\v{s}vili}
         {On a property of Urysohn's universal metric space}
         {\jRN{DANSSSR}}{101}{1955}{607--610 (Russian)}{#1}}
   \ITEE{#3}{JEHumphreys1990}{
      \BIb{#2}{J.E. Humphreys}
         {Reflection Groups and Coxeter Groups}
         {Cambridge University Press}{1990}{#1}}
   \ITEE{#3}{JRIsbell1964}{
      \BIB{#2}{J.R. Isbell}
         {Six theorems about injective metric spaces}
         {\jRN{CMHelv}}{39}{1964}{65--76}{#1}}
   \ITEE{#3}{SIzumino,YKato1985}{
      \BIB{#2}{S. Izumino and Y. Kato}
         {The closure of invertible operators on Hilbert space}
         {\jRN{ActaSM}}{49}{1985}{321--327}{#1}}
   \ITEE{#3}{CJiang2004}{
      \BIB{#2}{C. Jiang}
         {Similarity classification of Cowen-Douglas operators}
         {\jRN{CanadJM}}{56}{2004}{742--775}{#1}}
   \ITEE{#3}{WBJohnson,JLindenstrauss2001}{
      \BiB{#2}{W.B. Johnson and J. Lindenstrauss}{Basic Concepts in the Geometry of Banach Spaces}
         {Chapter 1 in:}{Handbook of the Geometry of Banach Spaces, Vol. 1}
         {W.B. Johnson and J. Lindenstrauss (editors), Elsevier Science B.V., Amsterdam}{2001}{1--84}{#1}}
   \ITEE{#3}{IBJung,JStochel2008}{
      \BIB{#2}{I.B. Jung and J. Stochel}
         {Subnormal operators whose adjoints have rich point spectrum}
         {\jRN{JFA}}{255}{2008}{1797--1816}{#1}}
   \ITEE{#3}{RVKadison,JRRingrose1983}{
      \BIb{#2}{R.V. Kadison and J.R. Ringrose}
         {Fundamentals of the Theory of Operator Algebras. Volume I: Elementary Theory}
         {Academic Press, Inc., New York-London}{1983}{#1}}
   \ITEE{#3}{RVKadison,JRRingrose1986}{
      \BIb{#2}{R.V. Kadison and J.R. Ringrose}
         {Fundamentals of the Theory of Operator Algebras. Volume II: Advanced Theory}
         {Academic Press, Inc., Orlando-London}{1986}{#1}}
   \ITEE{#3}{SKakutani1936}{
      \BIB{#2}{S. Kakutani}
         {\"{U}ber die Metrisation der topologischen Gruppen}
         {\jRN{ProcImpAcadTokyo}}{12}{1936}{82--84}{#1}}
   \ITEE{#3}{SKakutani1938}{
      \BIB{#2}{S. Kakutani}
         {Two fixed-point theorems concerning bicompact convex sets}
         {\jRN{ProcImpAcadTokyo}}{14}{1938}{242--245}{#1}}
   \ITEE{#3}{SKakutani1941}{
      \BIB{#2}{S. Kakutani}
         {Concrete representation of abstract L-spaces}
         {\jRN{AnnM}}{42}{1941}{523--537}{#1}}
   \ITEE{#3}{SKakutani1941a}{
      \BIB{#2}{S. Kakutani}
         {Concrete representation of abstract M-spaces}
         {\jRN{AnnM}}{42}{1941}{994--1024}{#1}}
   \ITEE{#3}{NKalton2007}{
      \BIB{#2}{N. Kalton}
         {Extending Lipschitz maps into $\CCc(K)$-spaces}
         {\jRN{IsraelJM}}{162}{2007}{275--315}{#1}}
   \ITEE{#3}{RKane2001}{
      \BIb{#2}{R. Kane}
         {Reflection Groups and Invariant Theory}
         {Canadian Mathematical Society, Springer}{2001}{#1}}
   \ITEE{#3}{VKannan,SRRaju1980}{
      \BIB{#2}{V. Kannan and S.R. Raju}
         {The nonexistence of invariant universal measures on semigroups}
         {\jRN{PAMS}}{78}{1980}{482--484}{#1}}
   \ITEE{#3}{IKaplansky1951}{
      \BIB{#2}{I. Kaplansky}
         {A theorem on rings of operators}
         {\jRN{PacJM}}{1}{1951}{227--232}{#1}}
   \ITEE{#3}{MKatetov1988}{
      \BiB{#2}{M. Kat\v{e}tov}{On universal metric spaces}{in: Frolik (ed.),}
         {General Topology and its Relations to Modern Analysis and Algebra VI. Proceedings of the Sixth Prague 
         Topological Symposium 1986}{Heldermann Verlag Berlin}{1988}{323--330}{#1}}
   \ITEE{#3}{YKatznelson1960}{
      \BIB{#2}{Y. Katznelson}
         {Sur les alg\'{e}bres dont les \'{e}l\'{e}ments non n\'{e}gatifs admettent des racines carr\'{e}es}
         {\jRN{AnnSciEcNormSupT}}{77}{1960}{167--174}{#1}}
   \ITEE{#3}{OHKeller1931}{
      \BIB{#2}{O.H. Keller}
         {Die Homoiomorphie der kompakten konvexen Mengen in Hilbertschen Raum}
         {\jRN{MAnn}}{105}{1931}{748--758}{#1}}
   \ITEE{#3}{MAKhamsi,WAKirk,CMartinez2000}{
      \BIB{#2}{M.A. Khamsi, W.A. Kirk, C. Martinez}
         {Fixed point and selection theorems in hyperconvex spaces}
         {\jRN{PAMS}}{128}{2000}{3275--3283}{#1}}
   \ITEE{#3}{ABKhararazishvili1998}{
      \BIb{#2}{A.B. Khararazishvili}
         {Transformation groups and invariant measures. Set-theoretic aspects}
         {World Scientific Publishing Co., Inc., River Edge, NJ}{1998}{#1}}
   \ITEE{#3}{YKijima1987}{
      \BIB{#2}{Y. Kijima}
         {Fixed points of nonexpansive self-maps of a compact metric space}
         {\jRN{JMAnApp}}{123}{1987}{114--116}{#1}}
  \ITEE{#3}{JSKim,ChRKim,SGLee1980}{
      \BIB{#2}{J.S. Kim, Ch.R. Kim, S.G. Lee}
         {Reducing operator valued spectra of a Hilbert space operator}
         {\jRN{JKoreanMS}}{17}{1980}{123--129}{#1}}
   \ITEE{#3}{JKindler1995}{
      \BIB{#2}{J. Kindler}
         {Minimax theorems with applications to convex metric spaces}
         {\jRN{CollM}}{68}{1995}{179--186}{#1}}
   \ITEE{#3}{WAKirk1998}{
      \BIB{#2}{W.A. Kirk}
         {Hyperconvexity of $\RRR$-trees}
         {\jRN{FM}}{156}{1998}{67--72}{#1}}
   \ITEE{#3}{VLKleeJr1952}{
      \BIB{#2}{V.L. Klee Jr.}
         {Invariant metrics in groups (solution of a problem of Banach)}
         {\jRN{PAMS}}{3}{1952}{484--487}{#1}}
   \ITEE{#3}{HJKowalsky1957}{
      \BIB{#2}{H.J. Kowalsky}
         {Einbettung metrischer R\"{a}ume}
         {\jRN{ArchM}}{8}{1957}{336--339}{#1}}
   \ITEE{#3}{WKubis,MRubin2010}{
      \BIB{#2}{W. Kubi\'{s} and M. Rubin}
         {Extension and reconstruction theorems for the Urysohn universal metric space}
         {\jRN{CzMJ}}{60}{2010}{1--29}{#1}}
   \ITEE{#3}{KKuratowski1966}{
      \BIb{#2}{K. Kuratowski}
         {Topology. \textup{Vol. I}}
         {\jRN{PWN}}{1966}{#1}}
   \ITEE{#3}{KKuratowski,BKnaster1927}{
      \BIB{#2}{K. Kuratowski and B. Knaster}
         {A connected and connected im kleinen point set which contains no perfect subset}
         {\jRN{BAMS}}{33}{1927}{106--109}{#1}}
   \ITEE{#3}{KKuratowski,AMostowski1976}{
      \BIb{#2}{K. Kuratowski and A. Mostowski}
         {Set Theory with an Introduction to Descriptive Set Theory}
         {\jRN{PWN}}{1976}{#1}}
   \ITEE{#3}{GLewicki1992}{
      \BIB{#2}{G. Lewicki}
         {Bernstein's ``lethargy'' theorem in metrizable topological linear spaces}
         {\jRN{MonatM}}{113}{1992}{213--226}{#1}}
   \ITEE{#3}{ASLewis1996}{
      \BIB{#2}{A.S. Lewis}
         {Group invariance and convex matrix analysis}
         {\jRN{SIAMJMAA}}{17}{1996}{927--949}{#1}}
   \ITEE{#3}{C-KLi,N-KTsing1991}{
      \BIB{#2}{C.-K. Li and N.-K. Tsing}
         {$G$-invariant norms and $G(c)$-radii}
         {\jRN{LAA}}{150}{1991}{179--194}{#1}}
   \ITEE{#3}{AJLazar,JLindenstrauss1971}{
      \BIB{#2}{A.J. Lazar and J. Lindenstrauss}
         {Banach spaces whose duals are $L_1$ spaces and their representing matrices}
         {\jRN{ActaM}}{126}{1971}{165--193}{#1}}
   \ITEE{#3}{EHLieb,MLoss1997}{
      \BIb{#2}{E.H. Lieb and M. Loss}
         {Analysis \textup{(Graduate Studies in Mathematics, vol. 14)}}
         {Amer. Math. Soc., Providence, RI}{1997}{#1}}
   \ITEE{#3}{ALindenbaum1926}{
      \BIB{#2}{A. Lindenbaum}
         {Contributions \`{a} l'\'{e}tude de l'espace m\'{e}trique I}
         {\jRN{FM}}{8}{1926}{209--222}{#1}}
   \ITEE{#3}{DLindenstrauss,LTzafriri1971}{
      \BIB{#2}{D. Lindenstrauss and L. Tzafriri}
         {On the complemented subspaces problem}
         {\jRN{IsraelJM}}{9}{1971}{263--269}{#1}}
   \ITEE{#3}{RILoebl1986}{
      \BIB{#2}{R.I. Loebl}
         {A note on containment of operators}
         {\jRN{BAustrMS}}{33}{1986}{279--291}{#1}}
   \ITEE{#3}{LHLoomis1945}{
      \BIB{#2}{L.H. Loomis}
         {Abstract congruence and the uniqueness of Haar measure}
         {\jRN{AnnM}}{46}{1945}{348--355}{#1}}
   \ITEE{#3}{LHLoomis1949}{
      \BIB{#2}{L.H. Loomis}
         {Haar measure in uniform structures}
         {\jRN{DukeMJ}}{16}{1949}{193--208}{#1}}
   \ITEE{#3}{ERLorch1939}{
      \BIB{#2}{E.R. Lorch}
         {Bicontinuous linear transformation in certain vector spaces}
         {\jRN{BAMS}}{45}{1939}{564--569}{#1}}
   \ITEE{#3}{ATLundell,SWeingram1969}{
      \BIb{#2}{A.T. Lundell and S. Weingram}
         {The topology of CW-complexes}
         {Litton Educ. Publ.}{1969}{#1}}
   \ITEE{#3}{WLusky1976}{
      \BIB{#2}{W. Lusky}
         {The Gurarij spaces are unique}
         {\jRN{ArchM}}{27}{1976}{627--635}{#1}}
   \ITEE{#3}{WLusky1977}{
      \BIB{#2}{W. Lusky}
         {On separable Lindenstrauss spaces}
         {\jRN{JFA}}{26}{1977}{103--120}{#1}}
   \ITEE{#3}{DMaharam1942}{
      \BIB{#2}{D. Maharam}
         {On homogeneous measure algebras}
         {\jRN{PNatlUSA}}{28}{1942}{108--111}{#1}}
   \ITEE{#3}{MMalicki,SSolecki2009}{
      \BIB{#2}{M. Malicki and S. Solecki}
         {Isometry groups of separable metric spaces}
         {\jRN{MProcCambPhS}}{146}{2009}{67--81}{#1}}
   \ITEE{#3}{PMankiewicz1972}{
      \BIB{#2}{P. Mankiewicz}
         {On extension of isometries in normed linear spaces}
         {\jRN{BAPolSSSM}}{20}{1972}{367--371}{#1}}
   \ITEE{#3}{JMartinezMaurica,MTPellon1987}{
      \BIB{#2}{J. Martinez-Maurica and M.T. Pell\'{o}n}
         {Non-archimedean Chebyshev centers}
         {\jRN{IndagMP}}{90}{1987}{417--421}{#1}}
   \ITEE{#3}{KMaurin1980}{
      \BIb{#2}{K. Maurin}
         {Analysis, Part II}
         {D. Reidel, Dordrecht-Boston-London}{1980}{#1}}
   \ITEE{#3}{SMazur,SUlam1932}{
      \BIB{#2}{S. Mazur and S. Ulam}
         {Sur les transformationes isom\'{e}triques d'espaces vectoriels norm\'{e}s}
         {\jRN{CRASParis}}{194}{1932}{946--948}{#1}}
   \ITEE{#3}{SMazurkiewicz1920}{
      \BIB{#2}{S. Mazurkiewicz}
         {Sur les lignes de Jordan}
         {\jRN{FM}}{1}{1920}{166--209}{#1}}
   \ITEE{#3}{SMazurkiewicz,WSierpinski1920}{
      \BIB{#2}{S. Mazurkiewicz and W. Sierpi\'{n}ski}
         {Contributions a la topologie des ensembles denombrables}
         {\jRN{FM}}{1}{1920}{17--27}{#1}}
   \ITEE{#3}{MMbekhta1992}{
      \BIB{#2}{M. Mbekhta}
         {Sur la structure des composantes connexes semi-Fredholm de $B(H)$}
         {\jRN{PAMS}}{116}{1992}{521--524}{#1}}
   \ITEE{#3}{JEMcCarthy1996}{
      \BIB{#2}{J.E. McCarthy}
         {Boundary values and Cowen-Douglas curvature}
         {\jRN{JFA}}{137}{1996}{1--18}{#1}}
   \ITEE{#3}{JMelleray2007}{
      \BIB{#2}{J. Melleray}
         {Computing the complexity of the relation of isometry between separable Banach spaces}
         {\jRN{MLQ}}{53}{2007}{128--131}{#1}}
   \ITEE{#3}{JMelleray2007a}{
      \BIB{#2}{J. Melleray}
         {On the geometry of Urysohn's universal metric space}
         {\jRN{TopA}}{154}{2007}{384--403}{#1}}
   \ITEE{#3}{JMelleray2008}{
      \BIB{#2}{J. Melleray}
         {Some geometric and dynamical properties of the Urysohn space}
         {\jRN{TopA}}{155}{2008}{1531--1560}{#1}}
   \ITEE{#3}{JMelleray,FVPetrov,AMVershik2008}{
      \BIB{#2}{J. Melleray, F.V. Petrov, A.M. Vershik}
         {Linearly rigid metric spaces and the embedding problem}
         {\jRN{FM}}{199}{2008}{177--194}{#1}}
   \ITEE{#3}{EMichael1953}{
      \BIB{#2}{E. Michael}
         {Some extension theorems for continuous functions}
         {\jRN{PacJM}}{3}{1953}{789--806}{#1}}
   \ITEE{#3}{EMichael1954}{
      \BIB{#2}{E. Michael}
         {Local properties of topological spaces}
         {\jRN{DukeMJ}}{21}{1954}{163--171}{#1}}
   \ITEE{#3}{EMichael1956}{
      \BIB{#2}{E. Michael}
         {Selected selection theorems}
         {\jRN{AmMMon}}{58}{1956}{233--238}{#1}}
   \ITEE{#3}{EMichael1956a}{
      \BIB{#2}{E. Michael}
         {Continuous selections. I}
         {\jRN{AnnM}}{63}{1956}{361--382}{#1}}
   \ITEE{#3}{EMichael1956b}{
      \BIB{#2}{E. Michael}
         {Continuous selections. II}
         {\jRN{AnnM}}{64}{1956}{562--580}{#1}}
   \ITEE{#3}{EMichael1959}{
      \BIB{#2}{E. Michael}
         {A theorem on semi-continuous set-valued functions}
         {\jRN{DukeMJ}}{26}{1959}{647--652}{#1}}
   \ITEE{#3}{JVanMill1986}{
      \BIB{#2}{J. van Mill}
         {Another counterexample in ANR theory}
         {\jRN{PAMS}}{97}{1986}{136--138}{#1}}
   \ITEE{#3}{JVanMill2001}{
      \BIb{#2}{J. van Mill}
         {The Infinite-Dimensional Topology of Function Spaces 
         \textup{(North-Holland Mathematical Library, vol. 64)}}
         {Elsevier, Amsterdam}{2001}{#1}}
   \ITEE{#3}{WMlak1991}{
      \BIb{#2}{W. Mlak}
         {Hilbert Spaces and Operator Theory}
         {PWN --- Polish Scientific Publishers and Kluwer Academic Publishers, Warszawa-Dordrecht}{1991}{#1}}
   \ITEE{#3}{JMogilski1979}{
      \BIB{#2}{J. Mogilski}
         {$CE$-decomposition of $l_2$-manifolds}
         {\jRN{BAPolSSSM}}{27}{1979}{309--314}{#1}}
   \ITEE{#3}{RLMoore1916}{
      \BIB{#2}{R.L. Moore}
         {On the foundations of plane analysis situs}
         {\jRN{TAMS}}{17}{1916}{131--164}{#1}}
   \ITEE{#3}{KMorita1955}{
      \BIB{#2}{K. Morita}
         {A condition for the metrizability of topological spaces and for $n$-dimensionality}
         {\jRN{SciRepTokyoA}}{5}{1955}{33--36}{#1}}
   \ITEE{#3}{AMukherjea,NATserpes1976}{
      \BIb{#2}{A. Mukherjea and N.A. Tserpes}
         {Measures on topological semigroups}
         {Springer Lecture Notes in Math. Vol. 547, Berlin}{1976}{#1}}
   \ITEE{#3}{JMycielski1974}{
      \BIB{#2}{J. Mycielski}
         {Remarks on invariant measures in metric spaces}
         {\jRN{CollM}}{32}{1974}{105--112}{#1}}
   \ITEE{#3}{SNNaboko1984}{
      \BIB{#2}{S.N. Naboko}
         {Conditions for similarity to unitary and selfadjoint operators}
         {\jRN{FunkAnalPril}}{18}{1984}{16--27}{#1}}
   \ITEE{#3}{LNachbin1965}{
      \BIb{#2}{L. Nachbin}
         {The Haar Integral}
         {D. Van Nostrand Company, Inc., Princeton-New Jersey-Toronto-New York-London}{1965}{#1}}
   \ITEE{#3}{TDNarang,SKGarg1991}{
      \BIB{#2}{T.D. Narang and S.K. Garg}
         {On the uniqueness of best approximation in non-archimedian spaces}
         {\jRN{PeriodMHung}}{22}{1991}{121--124}{#1}}
   \ITEE{#3}{JVonNeumann1930}{
      \BIB{#2}{J. von Neumann}
         {Zur Algebra der Funktionaloperationen und Theorie der normalen Operatoren}
         {\jRN{MAnn}}{102}{1930}{370--427}{#1}}
   \ITEE{#3}{JVonNeumann1934}{
      \BIB{#2}{J. von Neumann}
         {Zum Haarschen Mass in topologischen Gruppen}
         {\jRN{ComposM}}{1}{1934}{106--114}{#1}}
   \ITEE{#3}{JVonNeumann1937}{
      \BiB{#2}{J. von Neumann}
         {Some matrix-inequalities and metrization of matrix-space}{\jRN{TomskUnivRev}{} \textbf{1} (1937), 286--300; 
         in }{Collected Works}{Pergamon, New York}{1962}{Vol. 4, 205--219}{#1}}
   \ITEE{#3}{JVonNeumann1949}{
      \BIB{#2}{J. von Neumann}
         {On Rings of Operators. Reduction Theory}
         {\jRN{AnnM}}{50}{1949}{401--485}{#1}}
   \ITEE{#3}{ONielson1973}{
      \BIB{#2}{O. Nielson}
         {Borel sets of von Neumann algebras}
         {\jRN{AmJM}}{95}{1973}{145--164}{#1}}
   \ITEE{#3}{pn2002}{\bibITEM{#2}{#1} \mypaplist{pn1}}
   \ITEE{#3}{pn2006a}{\bibITEM{#2}{#1} \mypaplist{pn2}}
   \ITEE{#3}{pn2006b}{\bibITEM{#2}{#1} \mypaplist{pn3}}
   \ITEE{#3}{pn2007}{\bibITEM{#2}{#1} \mypaplist{pn4}}
   \ITEE{#3}{pn2008a}{\bibITEM{#2}{#1} \mypaplist{pn5}}
   \ITEE{#3}{pn2008b}{\bibITEM{#2}{#1} \mypaplist{pn6}}
   \ITEE{#3}{pn2009a}{\bibITEM{#2}{#1} \mypaplist{pn7}}
   \ITEE{#3}{pn2009b}{\bibITEM{#2}{#1} \mypaplist{pn8}}
   \ITEE{#3}{pn2009c}{\bibITEM{#2}{#1} \mypaplist{pn9}}
   \ITEE{#3}{pn2010a}{\bibITEM{#2}{#1} \mypaplist{pn12}}
   \ITEE{#3}{pn2010b}{\bibITEM{#2}{#1} \mypaplist{pn13}}
   \ITEE{#3}{pn2011a}{\bibITEM{#2}{#1} \mypaplist{pn10}}
   \ITEE{#3}{pn2011b}{\bibITEM{#2}{#1} \mypaplist{pn15}}
   \ITEE{#3}{pn2011c}{\bibITEM{#2}{#1} \mypaplist{pn16}}
   \ITEE{#3}{pn2011d}{\bibITEM{#2}{#1} \mypaplist{pn17}}
   \ITEE{#3}{pn2009x}{
      \bibITEM{#2}{#1} \mypaplist{pn11}}
   \ITEE{#3}{pn2010x}{
      \bibITEM{#2}{#1} \mypaplist{pn14}}
   \ITEE{#3}{pnXXXXb}{
      \bibITEM{#2}{#1} \mypaplist{pnX2}}
   \ITEE{#3}{pnXXXXc}{
      \bibITEM{#2}{#1} \mypaplist{pnX3}}
   \ITEE{#3}{pnXXXXd}{
      \bibITEM{#2}{#1} \mypaplist{pnX13}}
   \ITEE{#3}{MNiezgoda1998}{
      \BIB{#2}{M. Niezgoda}
         {Group majorization and Schur type inequalities}
         {\jRN{LAA}}{268}{1998}{9--30}{#1}}
   \ITEE{#3}{MNiezgoda1998a}{
      \BIB{#2}{M. Niezgoda}
         {An analytical characterization of effective and of irreducible groups inducing cone orderings}
         {\jRN{LAA}}{269}{1998}{105--114}{#1}}
   \ITEE{#3}{MNiezgoda,TYTam2001}{
      \BIB{#2}{M. Niezgoda and T.Y. Tam}
         {On norm property of $G(c)$-radii and Eaton triples}
         {\jRN{LAA}}{336}{2001}{119--130}{#1}}
   \ITEE{#3}{APazy1983}{
      \BIb{#2}{A. Pazy}{Semigroups of Linear Operators 
         and Applications to Partial Differential Equations \textup{(Applied Mathematical Sciences, vol. 44)}}
         {Springer-Verlag, New York}{1983}{#1}}
   \ITEE{#3}{APelc1982}{
      \BIB{#2}{A. Pelc}
         {Semiregular invariant measures on abelian groups}
         {\jRN{PAMS}}{86}{1982}{423--426}{#1}}
   \ITEE{#3}{RPenrose1955}{
      \BIB{#2}{R. Penrose}
         {A generalized inverse for matrices}
         {\jRN{ProcCambPhS}}{51}{1955}{406--413}{#1}}
   \ITEE{#3}{VPestov2006}{
      \BIb{#2}{V. Pestov}
         {Dynamics of infinite-dimensional groups. The Ramsey-Dvoretzky-Milman phenomenon}
         {University Lecture Series \textbf{40}, AMS, Providence, RI}{2006}{#1}}
   \ITEE{#3}{VPestov2007}{
      \BiB{#2}{V. Pestov}
         {Forty-plus annotated questions about large topological groups}
         {in:}{Open Problems in Topology II}{Elliot Pearl (editor), Elsevier B.V., Amsterdam}{2007}{439--450}{#1}}
   \ITEE{#3}{PVPetersen1993}{
      \BiB{#2}{P.V. Petersen}
         {Gromov-Hausdorff convergence of metric spaces}{in book:}{Differential Geometry: Riemannian Geometry 
         (Los Angeles, CA, 1990)}{Amer. Math. Soc., Providence, RI}{1993}{489--504}{#1}}
   \ITEE{#3}{DRamachandran,MMisiurewicz1982}{
      \BIB{#2}{D. Ramachandran and M. Misiurewicz}
         {Hopf's theorem on invariant measures for a group of transformations}
         {\jRN{SM}}{74}{1982}{183--189}{#1}}
   \ITEE{#3}{JMRosenblatt1974}{
      \BIB{#2}{J.M. Rosenblatt}
         {Equivalent invariant measures}
         {\jRN{IsraelJM}}{17}{1974}{261--270}{#1}}
   \ITEE{#3}{HLRoyden1963}{
      \BIb{#2}{H.L. Royden}
         {Real Analysis}
         {The Macmillan Co., New York}{1963}{#1}}
   \ITEE{#3}{WRudin1962}{
      \BIb{#2}{W. Rudin}
         {Fourier Analysis on Groups \textup{(Interscience Tracts in Pure and Applied Mathematics, Number 12)}}
         {Interscience Publishers, New York}{1962}{#1}}
   \ITEE{#3}{WRudin1991}{
      \BIb{#2}{W. Rudin}
         {Functional Analysis}
         {McGraw-Hill Science}{1991}{#1}}
   \ITEE{#3}{TSaito1972}{
      \BiB{#2}{T. Sait\^{o}}{Generations of von Neumann algebras}
         {Lecture Notes in Math. vol. 247}{\textup{(}Lecture on Operator Algebras\textup{)}}
         {Springer, Berlin-Heidelberg-New York}{1972}{435--531}{#1}}
   \ITEE{#3}{KSakai,MYaguchi2003}{
      \BIB{#2}{K. Sakai and M. Yaguchi}
         {Characterizing manifolds modeled on certain dense subspaces of non-separable Hilbert spaces}
         {\jRN{TsukubaJM}}{27}{2003}{143--159}{#1}}
   \ITEE{#3}{SSakai1971}{
      \BIb{#2}{S. Sakai}
         {$\CCc^*$-Algebras and $\WWw^*$-Algebras}
         {Springer-Verlag, Berlin-Heidelberg-New York}{1971}{#1}}
   \ITEE{#3}{RSchori1971}{
      \BIB{#2}{R. Schori}
         {Topological stability for infinite-dimensional manifolds}
         {\jRN{ComposM}}{23}{1971}{87--100}{#1}}
   \ITEE{#3}{JTSchwartz1967}{
      \BIb{#2}{J.T. Schwartz}
         {$\WWw^*$-algebras}
         {Gordon and Breach, Science Publishers Inc., New York-London-Paris}{1967}{#1}}
   \ITEE{#3}{ZSemadeni1971}{
      \BIb{#2}{Z. Semadeni}
         {Banach Spaces of Continuous Functions (Vol. I)}
         {\jRN{PWN}}{1971}{#1}}
   \ITEE{#3}{JPSerre1951}{
      \BIB{#2}{J.-P. Serre}
         {Homologie singuli\`{e}re des espaces fibr\'{e}s}
         {\jRN{AnnM}}{54}{1951}{425--505}{#1}}
   \ITEE{#3}{DSherman2007}{
      \BIB{#2}{D. Sherman}
         {On the dimension theory of von Neumann algebras}
         {\jRN{MScand}}{101}{2007}{123--147}{#1}}
   \ITEE{#3}{WSierpinski1928}{
      \BIB{#2}{W. Sierpi\'{n}ski}
         {Sur les projections des ensembles compl\'{e}mentaires aux ensembles \textup{(A)}}
         {\jRN{FM}}{11}{1928}{117--122}{#1}}
   \ITEE{#3}{MSlocinski1980}{
      \BIB{#2}{M. S\l{}oci\'{n}ski}
         {On the Wold-type decomposition of a pair of commuting isometries}
         {\jRN{APM}}{37}{1980}{255--262}{#1}}
   \ITEE{#3}{RCSteinlage1975}{
      \BIB{#2}{R.C. Steinlage}
         {On Haar measure in locally compact $T_2$ spaces}
         {\jRN{AmJM}}{97}{1975}{291--307}{#1}}
   \ITEE{#3}{JStochel,FHSzafraniec1989}{
      \BIB{#2}{J. Stochel and F.H. Szafraniec}
         {On normal extensions of unbounded operators. III. Spectral properties}
         {\jRN{PublRIMSKyoto}}{25}{1989}{105--139}{#1}}
   \ITEE{#3}{JStochel,FHSzafraniec1989a}{
      \BIB{#2}{J. Stochel and F.H. Szafraniec}
         {The normal part of an unbounded operator}
         {\jRN{ProcKonink}}{92}{1989}{495--503}{#1}}
   \ITEE{#3}{AHStone1962}{
      \BIB{#2}{A.H. Stone}
         {Absolute $\FFf_{\sigma}$-spaces}
         {\jRN{PAMS}}{13}{1962}{495--499}{#1}}
   \ITEE{#3}{AHStone1962a}{
      \BIB{#2}{A.H. Stone}
         {Non-separable Borel sets}
         {\jRN{DissM}}{28}{1962}{41 pages}{#1}}
   \ITEE{#3}{AHStone1972}{
      \BIB{#2}{A.H. Stone}
         {Non-separable Borel sets II}
         {\jRN{GTopA}}{2}{1972}{249--270}{#1}}
   \ITEE{#3}{MHStone1937}{
      \BIB{#2}{M.H. Stone}
         {Application of the theory of Boolean rings to general topology}
         {\jRN{TAMS}}{41}{1937}{375--481}{#1}}
   \ITEE{#3}{MHStone1948}{
      \BIB{#2}{M.H. Stone}
         {The generalized Weierstrass approximation theorem}
         {\jRN{MMag}}{21}{1948}{167--184}{#1}}
   \ITEE{#3}{BSz-Nagy1947}{
      \BIB{#2}{B. Sz.-Nagy}
         {On uniformly bounded linear transformations in Hilbert space}
         {\jRN{ActaSM}}{11}{1947}{152--157}{#1}}
   \ITEE{#3}{WTakahashi1970}{
      \BIB{#2}{W. Takahashi}
         {A convexity in metric space and nonexpansive mappings, I}
         {\jRN{KodaiMSemRep}}{22}{1970}{142--149}{#1}}
   \ITEE{#3}{MTakesaki2002}{
      \BIb{#2}{M. Takesaki}
         {Theory of Operator Algebras I \textup{(Encyclopaedia of Mathematical Sciences, Volume 124)}}
         {Springer-Verlag, Berlin-Heidelberg-New York}{2002}{#1}}
   \ITEE{#3}{MTakesaki2003}{
      \BIb{#2}{M. Takesaki}
         {Theory of Operator Algebras II \textup{(Encyclopaedia of Mathematical Sciences, Volume 125)}}
         {Springer-Verlag, Berlin-Heidelberg-New York}{2003}{#1}}
   \ITEE{#3}{MTakesaki2003a}{
      \BIb{#2}{M. Takesaki}
         {Theory of Operator Algebras III \textup{(Encyclopaedia of Mathematical Sciences, Volume 127)}}
         {Springer-Verlag, Berlin-Heidelberg-New York}{2003}{#1}}
   \ITEE{#3}{TYTam1999}{
      \BIB{#2}{T.Y. Tam}
         {An extension of a result of Lewis}
         {\jRN{ELA}}{5}{1999}{1--10}{#1}}
   \ITEE{#3}{TYTam2000}{
      \BIB{#2}{T.Y. Tam}
         {Group majorization, Eaton triples and numerical range}
         {\jRN{LMLA}}{47}{2000}{11--28}{#1}}
   \ITEE{#3}{TYTam2002}{
      \BIB{#2}{T.Y. Tam}
         {Generalized Schur-concave functions and Eaton triples}
         {\jRN{LMLA}}{50}{2002}{113--120}{#1}}
   \ITEE{#3}{TYTam,WCHill2001}{
      \BIB{#2}{T.Y. Tam and W.C. Hill}
         {On $G$-invariant norms}
         {\jRN{LAA}}{331}{2001}{101--112}{#1}}
   \ITEE{#3}{AFTiman,IAVestfrid1983}{
      \BIB{#2}{A.F. Timan and I.A. Vestfrid}
         {Any separable ultrametric space can be isometrically imbedded in $l_2$}
         {\jRN{FAA}}{17}{1983}{70--71}{#1}}
   \ITEE{#3}{JTomiyama1958}{
      \BIB{#2}{J. Tomiyama}
         {Generalized dimension function for $\WWw^*$-algebras of infinite type}
         {\jRN{TohokuMJ} (2)}{10}{1958}{121--129}{#1}}
   \ITEE{#3}{HTorunczyk1970}{
      \BIB{#2}{H. Toru\'{n}czyk}
         {Remarks on Anderson's paper ``On topological infinite deficiency''}
         {\jRN{FM}}{66}{1970}{393--401}{#1}}
   \ITEE{#3}{HTorunczyk1970a}{
      \BIb{#2}{H. Toru\'{n}czyk}
         {$G$-$K$-absorbing and skeletonized sets in metric spaces}
         {Ph.D. thesis, Inst. Math. Polish Acad. Sci., Warszawa}{1970}{#1}}
   \ITEE{#3}{HTorunczyk1972}{
      \BIB{#2}{H. Toru\'{n}czyk}
         {A short proof of Hausdorff's theorem on extending metrics}
         {\jRN{FM}}{77}{1972}{191--193}{#1}}
   \ITEE{#3}{HTorunczyk1974}{
      \BIB{#2}{H. Toru\'{n}czyk}
         {Absolute retracts as factors of normed linear spaces}
         {\jRN{FM}}{86}{1974}{53--67}{#1}}
   \ITEE{#3}{HTorunczyk1975}{
      \BIB{#2}{H. Toru\'{n}czyk}
         {On Cartesian factors and the topological classification of linear metric spaces}
         {\jRN{FM}}{88}{1975}{71--86}{#1}}
   \ITEE{#3}{HTorunczyk1978}{
      \BIB{#2}{H. Toru\'{n}czyk}
         {Concerning locally homotopy negligible sets and characterization of $l_2$-manifolds}
         {\jRN{FM}}{101}{1978}{93--110}{#1}}
   \ITEE{#3}{HTorunczyk1980}{
      \BiB{#2}{H. Toru\'{n}czyk}{Characterization of infinite-dimensional manifolds}{in:}
         {Proceedings of the International Conference on Geometric Topology (Warsaw, 1978)}
         {\jRN{PWN}}{1980}{431--437}{#1}}
   \ITEE{#3}{HTorunczyk1981}{
      \BIB{#2}{H. Toru\'{n}czyk}
         {Characterizing Hilbert space topology}
         {\jRN{FM}}{111}{1981}{247--262}{#1}}
   \ITEE{#3}{HTorunczyk1985}{
      \BIB{#2}{H. Toru\'{n}czyk}
         {A correction of two papers concerning Hilbert manifolds}
         {\jRN{FM}}{125}{1985}{89--93}{#1}}
   \ITEE{#3}{KTsuda1985}{
      \BIB{#2}{K. Tsuda}
         {A note on closed embeddings of finite dimensional metric spaces}
         {\jRN{BLondMS}}{17}{1985}{273--278}{#1}}
   \ITEE{#3}{PSUrysohn1925}{
      \BIB{#2}{P.S. Urysohn}
         {Sur un espace m\'{e}trique universel}
         {\jRN{CRASParis}}{180}{1925}{803--806}{#1}}
   \ITEE{#3}{PSUrysohn1927}{
      \BIB{#2}{P.S. Urysohn}
         {Sur un espace m\'{e}trique universel}
         {\jRN{BullSM}}{51}{1927}{43--64, 74--96}{#1}}
   \ITEE{#3}{VVUspenskij1986}{
      \BIB{#2}{V.V. Uspenskij}
         {A universal topological group with a countable basis}
         {\jRN{FAA}}{20}{1986}{86--87}{#1}}
   \ITEE{#3}{VVUspenskij1990}{
      \BIB{#2}{V.V. Uspenskij}
         {On the group of isometries of the Urysohn universal metric space}
         {\jRN{CMUC}}{31}{1990}{181--182}{#1}}
   \ITEE{#3}{VVUspenskij2004}{
      \BIB{#2}{V.V. Uspenskij}
         {The Urysohn universal metric space is homeomorphic to a Hilbert space}
         {\jRN{TopA}}{139}{2004}{145--149}{#1}}
   \ITEE{#3}{VVUspenskij2008}{
      \BIB{#2}{V.V. Uspenskij}
         {On subgroups of minimal topological groups}
         {\jRN{TopA}}{155}{2008}{1580--1606}{#1}}
   \ITEE{#3}{VSVaradarajan1963}{
      \BIB{#2}{V.S. Varadarajan}
         {Groups of automorphisms of Borel spaces}
         {\jRN{TAMS}}{109}{1963}{191--220}{#1}}
   \ITEE{#3}{AMVershik1998}{
      \BIB{#2}{A.M. Vershik}
         {The universal Urysohn space, Gromov's metric triples, and random metrics on the series of natural numbers}
         {\jRN{UspekhiMN}}{53}{1998}{57--64}{#1} English translation: \jRN{RussMS}{} \textbf{53} (1998), 921--928. 
         Correction: \jRN{UspekhiMN}{} \textbf{56} (2001), p. 207. English translation: \jRN{RussMS}{} \textbf{56} 
         (2001), p. 1015.}
   \ITEE{#3}{AMVershik2002}{
      \BIb{#2}{A.M. Vershik}
         {Random metric spaces and the universal Urysohn space}
         {Fundamental Mathematics Today. 10th anniversary of the Independent Moscow University. MCCME Publ.}{2002}{#1}}
   \ITEE{#3}{NWeaver1999}{
      \BIb{#2}{N. Weaver}
         {Lipschitz Algebras}
         {World Scientific}{1999}{#1}}
   \ITEE{#3}{JWeidmann1980}{
      \BIb{#2}{J. Weidmann}
         {Linear Operators in Hilbert Spaces}
         {(Graduate Texts in Mathematics, vol. 68) Springer-Verlag New York Inc.}{1980}{#1}}
   \ITEE{#3}{JEWest1969}{
      \BIB{#2}{J.E. West}
         {Approximating homotopies by isotopies in Fr\'{e}chet manifolds}
         {\jRN{BAMS}}{75}{1969}{1254--1257}{#1}}
   \ITEE{#3}{JEWest1969a}{
      \BIB{#2}{J.E. West}
         {Fixed-point sets of transformation groups on infinite-product spaces}
         {\jRN{PAMS}}{21}{1969}{575--582}{#1}}
   \ITEE{#3}{JEWest1970}{
      \BIB{#2}{J.E. West}
         {The ambient homeomorphy of infinite-dimensional Hilbert spaces}
         {\jRN{PacJM}}{34}{1970}{257--267}{#1}}
   \ITEE{#3}{JHCWhitehead1949}{
      \BIB{#2}{J.H.C. Whitehead}
         {Combinatorial homotopy I}
         {\jRN{BAMS}}{55}{1949}{213--245}{#1}}
   \ITEE{#3}{GTWhyburn1942}{
      \BIb{#2}{G. T. Whyburn}
         {Analytic Topology}
         {Amer. Math. Soc. Colloquium Publications (vol. XXVIII), New York}{1942}{#1}}
   \ITEE{#3}{WWogen1969}{
      \BIB{#2}{W. Wogen}
         {On generators for von Neumann algebras}
         {\jRN{BAMS}}{75}{1969}{95--99}{#1}}
   \ITEE{#3}{RYTWong1967}{
      \BIB{#2}{R.Y.T. Wong}
         {On homeomorphisms of certain infinite dimensional spaces}
         {\jRN{TAMS}}{128}{1967}{148--154}{#1}}
   \ITEE{#3}{LYang,JZhang1987}{
      \BIB{#2}{L. Yang and J. Zhang}
         {Average distance constants of some compact convex space}
         {\jRN{JChinUST}}{17}{1987}{17--23}{#1}}
   \ITEE{#3}{PZakrzewski1993}{
      \BIB{#2}{P. Zakrzewski}
         {The existence of invariant $\sigma$-finite measures for a group of transformations}
         {\jRN{IsraelJM}}{83}{1993}{275--287}{#1}}
   \ITEE{#3}{PZakrzewski2002}{
      \BIb{#2}{P. Zakrzewski}
         {Measures on Algebraic-Topological Structures, Handbook of Measure Thoery}
         {E. Pap, ed., Elsevier, Amsterdam}{2002, 1091--1130}{#1}}
   \ITEE{#3}{KZhu2000}{
      \BIB{#2}{K. Zhu}
         {Operators in Cowen-Douglas classes}
         {\jRN{IllinoisJM}}{44}{2000}{767--783}{#1}}
   }

\newcommand{\mypaplist}[2][]{
   \ITEE{#2}{pn1}{
      \myBIB{Separate and joint similarity to families of normal operators}
         {\jRN[#1]{SM}}{149}{2002}{39--62}}
   \ITEE{#2}{pn2}{
      \myBIB{Locally arcwise connected metrizable spaces with the fixed point property are complete-metrizable}
         {\jRN[#1]{TopA}}{153}{2006}{1639--1642}}
   \ITEE{#2}{pn3}{
      \myBIB{Invariant measures for equicontinuous semigroups of continuous transformations of a compact Hausdorff space}
         {\jRN[#1]{TopA}}{153}{2006}{3373--3382}}
   \ITEE{#2}{pn4}{
      \myBIB{Approximation of the Hausdorff distance by the distance of continuous surjections}
         {\jRN[#1]{TopA}}{154}{2007}{655--664}}
   \ITEE{#2}{pn5}{
      \myBIB{Generalized Haar integral}
         {\jRN[#1]{TopA}}{155}{2008}{1323--1328}}
   \ITEE{#2}{pn6}{
      \myBIB{Integration and Lipschitz functions}
         {\jRN[#1]{RCMP}}{57}{2008}{391--399}}
   \ITEE{#2}{pn7}{
      \myBIB{Canonical Banach function spaces generated by Urysohn universal spaces. Measures as Lipschitz maps}
         {\jRN[#1]{SM}}{192}{2009}{97--110}}
   \ITEE{#2}{pn8}{
      \myBIB{Urysohn universal spaces as metric groups of exponent $2$}
         {\jRN[#1]{FM}}{204}{2009}{1--6}}
   \ITEE{#2}{pn9}{
      \myBIB{Central subsets of Urysohn universal spaces}
         {\jRN[#1]{CMUC}}{50}{2009}{445--461}}
   \ITEE{#2}{pn10}{
      \myBIB[P. Niemiec and T.Y. Tam]{A representation of $G$-in\-variant norms for Eaton triple}
         {\jRN[#1]{JCA}}{18}{2011}{59--65}}
   \ITEE{#2}{pn11}{
      \myBIB{Functor of extension of contractions on Urysohn universal spaces}
         {\jRN[#1]{ACS}}{}{2009}{\texttt{DOI: 10.1007/s10485-009-9218-z}}}
   \ITEE{#2}{pn12}{
      \myBIB{Ultra-$\mM$-separability}
         {\jRN[#1]{TopA}}{157}{2010}{669--673}}
   \ITEE{#2}{pn13}{
      \myBIB{Functor of extension of $\Lambda$-isometric maps between central subsets 
         of the unbounded Urysohn universal space}{\jRN[#1]{CMUC}}{51}{2010}{541--549}}
   \ITEE{#2}{pn14}{
      \myBIB{Normed topological pseudovector groups}{\jRN[#1]{ACS}}{}{2010}
         {\ITE{\equal{#1}{}}{\texttt{DOI: 10.1007/s10485\-010-9239-7}}{\texttt{DOI: 10.1007/s10485-010-9239-7}}}}
   \ITEE{#2}{pn15}{
      \myBIB{Topological structure of Urysohn universal spaces}
         {\jRN[#1]{TopA}}{158}{2011}{352--359}}
   \ITEE{#2}{pn16}{
      \myBIB{A note on invariant measures}
         {\jRN[#1]{OpusM}}{31}{2011}{425--431}}
   \ITEE{#2}{pn17}{
      \myBIB{Strengthened Stone-Weierstrass type theorem}
         {\jRN[#1]{OpusM}}{31}{2011}{645--650}}
   \ITEE{#2}{pnX2}{
      \myBAPP{Functor of continuation in Hilbert cube and Hilbert space}
         {to appear in \jRN[#1]{FM}}}
   \ITEE{#2}{pnX3}{
      \myBAPP{Norm closures of orbits of bounded operators}
         {to appear.}}
   \ITEE{#2}{pnX6}{
      \myBAPP{Extending maps by injective $\sigma$-$Z$-maps in Hilbert manifolds}
         {to appear in \jRN[#1]{BullPol}}}
   \ITEE{#2}{pnX7}{
      \myBAPP{Spaces of measurable functions}
         {submitted to \jRN[#1]{CollectM}}}
   \ITEE{#2}{pnX8}{
      \myBAPP{Normal systems over ANR's, rigid embeddings and nonseparable absorbing sets}
         {submitted to \jRN[#1]{ActaMSinES}}}
   \ITEE{#2}{pnX9}{
      \myBAPP{Borel structure of the spectrum of a closed operator}
         {submitted to \jRN[#1]{SM}}}
   \ITEE{#2}{pnX10}{
      \myBAPP{Central points and measures and dense subsets of compact metric spaces}
         {submitted to \jRN[#1]{TopMethNA}}}
   \ITEE{#2}{pnX11}{
      \myBAPP{Generalized absolute values and polar decompositions of a bounded operator}
         {submitted to \jRN[#1]{IEOT}.}}
   \ITEE{#2}{pnX12}{
      \myBAPP{Ultrametrics, extending of Lipschitz maps and nonexpansive selections}
         {accepted for publication in \jRN[#1]{HJM}}}
   \ITEE{#2}{pnX13}{
      \myBAPP{A note on ANR's}
         {submitted to \jRN[#1]{TopA}}}
   \ITEE{#2}{pnX14}{
      \myBAPP{Problem with almost everywhere equality}
         {submitted to \jRN[#1]{ArchM}}}
   \ITEE{#2}{pnX15}{
      \myBAPP{Universal valued Abelian groups}
         {submitted to \jRN[#1]{LNM}}}
   \ITEE{#2}{pnX16}{
      \myBAPP{Unitary equivalence and decompositions of finite systems of closed densely defined operators 
         in Hilbert spaces}{submitted to \jRN[#1]{DissM}}}
   }

\begin{document}

\title[Nonseparable absorbing sets]{Normal systems over ANR's,\\rigid embeddings and\\nonseparable absorbing sets}
\myData
\begin{abstract}
Most of results of Bestvina and Mogilski [\textit{Characterizing certain incomplete infinite-dimensional absolute
retracts}, Michigan Math. J. \textbf{33} (1986), 291--313] on strong $Z$-sets in ANR's and absorbing sets is generalized
to nonseparable case. It is shown that if an ANR $X$ is locally homotopy dense embeddable in infinite-dimensional Hilbert
manifolds and $w(U) = w(X)$ (where `$w$' is the topological weight) for each open nonempty subset $U$ of $X$,
then $X$ itself is homotopy dense embeddable in a Hilbert manifold. It is also demonstrated that whenever $X$ is an AR,
its weak product $W(X,*) = \{(x_n)_{n=1}^{\infty} \in X^{\omega}\dd\ x_n = * \textup{ for almost all } n\}$
is homeomorphic to a pre-Hilbert space $E$ with $E \cong \Sigma E$. An intrinsic characterization of manifolds modelled
on such pre-Hilbert spaces is given.
\end{abstract}
\maketitle


In \cite{henderson} Henderson has shown that a $Z$-set in a paracompact manifold $M$ modelled on a metrizable
locally convex topological vector space $F$ such that $F^{\omega} \cong F$ is a strong $Z$-set in $M$. This result
was used by Chapman \cite{chapman} to generalize the results of Anderson and McCharen \cite{and-mcch} on extending
homeomorphisms between $Z$-sets of a manifold modelled on an infinite-dimensional Fr\'{e}chet space. The homeomorphism
extension theorem was applied in Toru\'{n}czyk's original proof \cite{tor1,tor2} that every Fr\'{e}chet space
is homeomorphic to a Hilbert space. In his proof also strong $Z$-sets play important role. In the meantime it turned
out that these sets are more applicable in the theory of incomplete ANR's than $Z$-sets. With use of strong
$Z$-sets several infinite-dimensional AR's have been characterized, see e.g. \cite{b-m}, \cite{brz}, \cite{bbmw},
\cite{dij1,dij2}, \cite{s-y}. Strong $Z$-sets are therefore an important tool in studying ANR's. We present here several
theorems on strong $Z$-sets in (nonseparable) ANR's which, in particular, generalize the results of Henderson
\cite{henderson} and Bestvina and Mogilski \cite{b-m} and we use them to generalize most important facts on absorbing sets
due to the latter authors. In their exposition and in \cite{brz} the second axiom of countability plays an important
role and one may suggest that it is the point. In case of nonseparable ANR's one has to use different methods
to prove that, e.g., being a strong $Z$-set is a local property. We show this by means of so-called \textit{normal
systems}, which turn out to sum up common features of the notion of a strong $Z$-set and the well-known strong discrete
approximation property (and similar ones characterizing nonseparable Hilbert manifolds \cite{tor1}). With use of normal
systems and the so-called \textit{small maps approximation property} (which discovers the hidden nature of normal systems),
in short: SMAP, we show that if an ANR $X$ is locally homotopy dense embeddable in Hilbert manifolds and all its nonempty
open subsets have the same topological weight, then $X$ itlself is homotopy dense embeddable in a Hilbert manifold as well.
SMAP for normal systems also enables us to shorten Toru\'{n}czyk's original proof of the Hilbert space manifold
characterization theorem, namely: Toru\'{n}czyk \cite[Proof of~3.2, p.~256]{tor1} in the final part of the proof
of the theorem characterizing separable Hilbert manifolds among complete ANR's by means of the strong discrete
approximation property (briefly, SDAP; cf. \cite{acm}) argued that if a separable complete ANR $X$ has SDAP, then $X$ is
locally a Hilbert manifold. However, he gave no explanation why SDAP is open hereditary, that is, if $X$ has SDAP,
then all its open subsets also have SDAP (note that the limitation topology of $\CCc(D,U)$ does not coincide
with the topology of a subspace induced by the limitation topology of $\CCc(D,X)$ if $U$ is open in $X$). We shall easily
see, thanks to SMAP for normal systems, that Toru\'{n}czyk's condition \cite[($*$2), p.~253]{tor1} is implied by SDAP.\par
The problem of investigation of nonseparable absorbing spaces was mentioned in the seminal paper \cite{d-mo} that greatly
stimulated the development of the classical (separable) theory of absorbing spaces. So, in a sense, the recent paper
resolves an old problem posed in the known list of problems \cite{d-mo}. Partial results in this direction were also
obtained in 2003 by Sakai and Yaguchi \cite{s-y}.\par
Other topic, discussed in the paper, is related to the problem of classification of the weak products of AR's
(or, equivalently, absorbing sets for topological, closed hereditary, additive classes $\CcC$ such that
$C_1 \times C_2 \in \CcC$ for all $C_1$, $C_2 \in \CcC$). We introduce \textit{rigid} embeddings into normed spaces
and by means of them we prove the main result of the paper which is new even in separable case and states that the weak
product (defined in Abstract) of an arbitrary AR is homeomorphic to a pre-Hilbert space $E$ such that $E \cong \Sigma E$.
This show that Corollary~5.4 of \cite{b-m} which naturally generalizes Toru\'{n}czyk's Factor Theorem \cite{tor4,tor3}
(cf. \cite{d-m}) is in fact equivalent to it. Finally, we give an intrinsic characterization of all nonzero pre-Hilbert
spaces $E$ with $E \cong \Sigma E$: a metrizable space $X$ is homeomorphic to such a space iff $X$ is an AR
and a $\sigma$-$Z$-space such that for each $Z$-set $K$ in $X$ the natural projection $(X \setminus K) \times X \to X
\setminus K$ is a near-homeomorphism (i.e. it is approximable, in the limitation topology, by homeomorphisms).

\SECT{Preliminaries}

In this paper $\NNN$, $I$ and $Q$ denote the set of all nonnegative integers, the unit interval $[0,1]$ and the Hilbert
cube (i.e. $Q \cong I^{\omega}$), respectively. The letters $X$, $Y$, $Z$, $K$, etc. stand for topological spaces.
Following Banakh and Zarichnyy \cite{b-z}, we identify cardinals with the sets of ordinals of smaller size and endow them
with the discrete topologies. By an \textit{ANR} we mean a metrizable space which is an absolute neighbourhood retract
for metrizable spaces. Compact and paracompact spaces are meant to be Hausdorff, in the opposite to normal spaces which
are understood by us as having the property of separating closed disjoint sets. We write $Y \cong Z$ iff $Y$ and $Z$
are homeomorphic. $X^{\omega}$ stands for the countable infinite Cartesian power of $X$, equipped with the Tichonov
topology, and $\cov(X)$ is used to denote the collection of all open covers of $X$. By a \textit{map} we mean a continuous
function. Whenever $g$ is a map, $\im g$ and $\overline{\im}\,g$ stand for, respectively, the image of $g$ and its
closure. If $A$ is a subset of $X$, $\intt A$ and $\bar{A}$ denote the interior and the closure of $A$ in the whole space
$X$. We use $w(X)$ to denote the topological weight of $X$.\par
If $Y$ is paracompact, the space $\CCc(X,Y)$ of all maps of $X$ into $Y$ in this paper is always equipped with
the \textit{limitation topology}. For definition and basic properties of this topology the Reader is referred
to \cite{tor1}, \cite{bowers}. The symbol $B(f,\UUu)$ (with $f \in \CCc(X,Y)$ and $\UUu \in \cov(Y)$) has the same
meaning as in \cite{tor1} and $B(f,\UUu)$ consists of all maps of $X$ to $Y$ which are \textit{$\UUu$-close to $f$}.\par
In the sequel we shall make use of the following powerful result.

\begin{lemm}{michael}{Michael \mbox{\cite{michael}}, cf. \mbox{\cite[Proposition~4.1]{b-p}}}
Let $X$ be a paracompact space and $\WwW$ a collection of some subsets of $X$ which satisfies the following three
conditions:
\begin{enumerate}[\upshape(M1)]
\item If $A \in \WwW$ and $U$ is an open subset of $X$ contained in $A$, then $U \in \WwW$.
\item If $U_1$ and $U_2$ are open subsets of $X$ and $U_1,U_2 \in \WwW$, then $U_1 \cup U_2 \in \WwW$.
\item If $\{U_s\}_{s \in S}$ is a discrete (in $X$) collection of open subsets of $X$ each of which is a member of $\WwW$,
   then $\bigcup_{s \in S} U_s \in \WwW$.
\end{enumerate}
Then, $X \in \WwW$ provided for every point $a$ of $X$ there is $A \in \WwW$ such that $a \in \intt A$.
\end{lemm}

For simplicity, every family (of subsets of a given topological space) which satisfies the properties (M1)--(M3)
we call a \textit{Michael collection}.\par
Recall that a space $E$ is said to be a \textit{neighbourhood extensor for a space $X$} iff every map from any closed
subset of $X$ into $E$ is extendable to a map defined on some open subset of $X$ (and with values in $E$).
If $E$ is a neighbourhood extensor for $X$, every open subset of $E$ is a neighbourhood extensor for each closed subset
of $X$, and $X$ is normal provided $E$ is Hausdorff and has more than one point. A space $Y$ is called
\textit{locally equiconnected} (in short: LEC) iff there is an open in $Y \times Y$ neighbourhood $\Omega$ of the diagonal
$\Delta_Y = \{(y,y)\dd\ y \in Y\}$ and a map $\lambda\dd \Omega \times I \to Y$ such that $\lambda(y,y,t) = y$,
$\lambda(x,y,0) = x$ and $\lambda(x,y,1) = y$ for each $(x,y) \in \Omega$ and $t \in I$. Such a map is called
an \textit{equiconnecting function} (\cite{dug2}). Every ANR is LEC and there are examples of separable completely
metrizable LEC spaces which are not ANR's (see e.g. \cite{cauty}). However, each LEC space is locally contractible
and finite dimensional locally contractible metrizable spaces are ANR's (\cite{dug1}).\par
In the next section we shall need the following two properties of neighbourhood extensors, the proofs of which are left
as exercises.

\begin{lem}{NE-LEC}
Let $X$ and $Y$ be normal spaces such that $Y$ is a neighbourhood extensor for $X$ and let $A$ be a closed subset of $X$.
\begin{enumerate}[\upshape(A)]
\item If $f\dd A \to Y$ is a map such that $\im f \subset V$ where $V$ is an open in $Y$ set contractible in its open
   neighbourhood $U \supset V$, then $f$ is extendable to a map of $X$ into $U$.
\item \textup{(cf. \cite[Lemma~1.3]{tor1})} If, additionally, $Y$ is a paracompact LEC space,
   then the map $\CCc(X,Y) \ni u \mapsto u\bigr|_A \in \CCc(A,Y)$ is open.
\end{enumerate}
\end{lem}

Whenever we talk about the (topological) dimension, we mean the covering one. If $\UUu$ is a family
of subsets of a space $X$, $\ord(\UUu)$ stands for the order of $\UUu$ and it is understood as a natural number
or $\infty$. We say that $X$ is \textit{of finite-dimensional type} (briefly, FDT) iff every open cover of $X$
(not necessarily finite) has a refinement (in $\cov(X)$) of finite order. $X$ is said to be \textit{locally FDT}
if every point of $X$ has a (not necessarily open) neighbourhood which is FDT. Important examples of FDT [locally FDT]
spaces are [locally] compact ones.\par
We denote by $\comp(X)$ the least \textit{infinite} cardinal $\alpha$ such that every open cover of $X$ has
a subcover of cardinality less than $\alpha$. Similarly, $\comp_l(X)$ is the least infinite cardinal $\alpha$
such that every point of $X$ has a (not necessarily open) neighbourhood $F$ such that $\comp(F) \leqsl \alpha$.
(Observe that $\comp(X) = \aleph_0$ [$\comp_l(X) = \aleph_0$] iff $X$ is [locally] compact.) The proofs of the following
results are omitted. (Recall that a \textit{discrete} subset of a topological space is a \textbf{closed} set whose topology
is discrete.)

\begin{lem}{comp}
Let $X$ be a paracompact space such that the set $X$ is infinite.
\begin{enumerate}[\upshape(I)]
\item $\comp(X)$ is the least cardinal $\alpha$ with the following property: for every locally finite open cover
   $\{U_s\}_{s \in S}$ of $X$ consisting of non\-empty sets, $\card S < \alpha$.
\item $\comp(X)$ is the least cardinal $\alpha$ such that every discrete subset of $X$ has cardinality
   less than $\alpha$.
\item If $X$ is metrizable, then either $X$ contains a discrete subset of cardinality $w(X)$ and then $\comp(X)$
   is the direct successor of $w(X)$, or each discrete subset of $X$ is of cardinality less than $w(X)$
   and then $\comp(X) = w(X)$. What is more, in the second case there is a sequence of cardinals $\alpha_0 < \alpha_1
   < \ldots$ such that $w(X) = \sup_{n\in\NNN} \alpha_n$.
\end{enumerate}
\end{lem}

\begin{lem}{omega}
If $X$ is metrizable and contains a closed set homeomorphic to $X \times \NNN$, then $X$ has a discrete subset
of cardinality $w(X)$.
\end{lem}

We shall also involve some properties (the same as were used in \cite{tor1}) of simplicial complexes
with Whitehead's weak topologies or the metric ones. By a (combinatorial) simplicial complex $\KKk$ we shall always mean
a complex whose vertices form an orthonormal system in some Hilbert space $\HHh$, and its geometric realization $|\KKk|$
will always be identified with the suitable subset of $\HHh$. If $|\KKk|$ is equipped with the weak topology, we shall
write $|\KKk|_w$. If it is equipped with the metric topology induced from the topology of $\HHh$, we shall write
$|\KKk|_m$. The map $|\KKk|_w \ni x \mapsto x \in |\KKk|_m$ is denoted by $j_{\KKk}$. Adapting Toru\'{n}czyk's
proof of \cite[Lemma~3.4]{tor1} (cf. \cite[Proof of Lemma~3.2]{dowker1}) one may show that

\begin{lem}{complex}
Let $X$ be a normal space, $Y$ be an ANR, $V$ an open subset of $Y$, $f\dd X \to V$ a map and let $\UUu \in \cov(Y)$.
There is a simplicial complex $\KKk$ and maps $v\dd X \to |\KKk|_w$ and $w\dd |\KKk|_m \to Y$ such that $w j_{\KKk} v
\in B(f,\UUu)$ and
\begin{enumerate}[\upshape(SC1)]
\item $\KKk$ is locally finite dimensional,
\item $\KKk$ has less than $\min(\comp(X),\comp(\bar{V}))$ vertices,
\item $\dim \KKk \leqsl \min (\dim(X),\dim(V))$,
\item $\KKk$ is finite dimensional provided $X$ is FDT or $\bar{V}$ is FDT,
\item $\KKk$ is locally finite provided $V$ is separable.
\end{enumerate}
\end{lem}

For more information on simplicial complexes see e.g. \cite{whitehead}, \cite{dowker2,dowker1}, \cite{l-w}
or \cite[II.\S6]{b-p}.

\SECT{Small maps approximation property}

We begin with

\begin{dfn}{small}
For a subset $B$ of a metrizable space $Y$, let $\SsS_Y(X,B)$ be the collection of all maps $g\dd X \to Y$
such that $\overline{\im}\,g \subset B$. Note that if $B$ is open in $Y$, $\SsS_Y(X,B)$ is open in $\CCc(X,Y)$.
Similarly, if $\BBb$ is a family of subsets of $Y$, $\SsS_Y(X,\BBb)$ stands for the union of all $\SsS_Y(X,B)$
with $B \in \BBb$. The members of $\SsS_Y(X,\BBb)$ are said to be \textit{$\BBb$-small (in $Y$) maps}. We write
\textit{$B$-small (in $Y$)} instead of $\{B\}$-small. (It would be more appropriate to say `strongly small'.)
\end{dfn}

\begin{dfn}{smap}
A subset $D$ of $\CCc(X,Y)$ (with paracompact $Y$) is said have \textit{small maps approximation property}
(in short: SMAP) iff there is $\UUu \in \cov(Y)$ such that $\SsS_Y(X,\UUu) \subset \bar{D}$ (the closure taken
in the limitation topology of $\CCc(X,Y)$).
\end{dfn}

We call a class $\TTt$ of topological spaces \textit{closed hereditary} (respectively \textit{open hereditary})
if $A \in \TTt$ for every closed (open) subset $A$ of any member of $\TTt$.\par
Utility of SMAP is explained in the following

\begin{thm}{smap}
Let $Y$ be a paracompact LEC space which is a neighbourhood extensor for a space $X$. Let $\TTt$ be the family
of all closed subsets of $X$. Suppose that $\{D_A\}_{A \in \TTt}$ is a collection such that
\begin{enumerate}[\upshape(D1)]\setcounter{enumi}{-1}
\item For each $A \in \TTt$, $D_A$ is an open subset of $\CCc(A,Y)$.
\item If $B \in \TTt$, $A$ is a closed subset of $B$ and $g \in D_B$, then $g\bigr|_A \in D_A$.
\item If $A \in \TTt$ is the union of its two closed subsets $A_1$ and $A_2$ and $g \in \CCc(A,Y)$ is such that
   $g\bigr|_{A_j} \in D_{A_j}$ for $j=1,2$, then $g \in D_A$.
\item If $A \in \TTt$ is the union of a discrete (in $A$) family $\{A_t\}_{t \in T}$ of its closed subsets
   and $g \in \CCc(A,X)$ is such that $g\bigr|_{A_s} \in D_{A_s}$ for each $s \in T$ and the family
   $\{g(A_t)\}_{t \in T}$ is discrete in $Y$, then $g \in D_A$.
\end{enumerate}
Then \tfcae
\begin{enumerate}[\upshape(i)]
\item each $D_A$ is a dense subset of $\CCc(A,Y)$,
\item $D_X$ has SMAP.
\end{enumerate}
\end{thm}
\begin{proof}
We may assume that $Y$ has more than one point. This implies that $X$ is normal. Basicly, we only need to show that (i)
is implied by (ii). Let $\WwW$ be the family of all open subsets $U$ of $Y$ such that $\SsS_Y(A,U) \subset \bar{D}_A$
for each $A \in \TTt$. We shall show that $\WwW$ is a Michael collection.
The point (M1) is clearly fulfilled and (M3) is left as an exercise. We pass to (M2). Let $U_1,U_2 \in \WwW$,
$U = U_1 \cup U_2$, $A \in \TTt$, $f\dd A \to Y$ be $U$-small in $Y$ and let $\VVv \in \cov(Y)$. Take a star refinement
$\GGg \in \cov(Y)$ of $\VVv$. Let $U^*$ be such an open subset of $Y$ that $\overline{\im} f \subset U^*$
and $\overline{U^*} \subset U$. Since the sets $\overline{U^*} \setminus U_1$ and $\overline{U^*} \setminus U_2$
are closed and disjoint, there are two open sets $U_1^*$ and $U_2^*$ for which $\overline{U^*} \setminus U_j
\subset U_j^*\ (j=1,2)$ and
\begin{equation}\label{eqn:aux200}
\overline{U_1^*} \cap \overline{U_2^*} = \varempty.
\end{equation}
Put $B_j = \overline{U^*} \setminus U_j^*\ (j=1,2)$ and note that $B_1$ and $B_2$ are closed subsets of $Y$ such that
$\overline{\im} f \subset B_1 \cup B_2$ and $B_j \subset U_j\ (j=1,2)$. Further, take an open set $U_0$ for which
$B_2 \subset U_0$ and $\overline{U}_0 \subset U_2$. Now put $A_j = f^{-1}(B_j) \in \TTt$, $f_1 = f\bigr|_{A_1}\dd
A_1 \to Y$ and $\GGg_1 = \{G \cap U_0\dd\ G \in \GGg\} \cup \{G \setminus B_2\dd\ G \in \GGg\} \in \cov(Y)$. Observe
that (by \eqref{eqn:aux200})
\begin{equation}\label{eqn:disjoint}
\overline{A \setminus A_1} \cap \overline{A \setminus A_2} = \varempty.
\end{equation}
By \LEM{NE-LEC}--(B), there is $\GGg_1' \in \cov(Y)$ such that
\begin{equation}\label{eqn:open}
B(f_1,\GGg_1') \subset B(f,\GGg_1)\bigr|_{A_1},
\end{equation}
that is, for each $\GGg_1'$-close to $f_1$ map $h_1\dd A_1 \to Y$ there is a $\GGg_1$-close to $f$ map $h\dd A \to Y$
which extends $h_1$. Since $U_1 \in \WwW$ and $\overline{\im} f_1 \subset U_1$, there is $g_1' \in D_{A_1}$ which
is $\GGg_1'$-close to $f_1$. Thanks to \eqref{eqn:open} we may find $g_1 \in \CCc(A,Y)$ which is $\GGg_1$-close to $f$
and extends $g_1$. This yields that
\begin{equation}\label{eqn:g1}
g_1 \in B(f,\GGg_1), \qquad g_1\bigr|_{A_1} \in D_{A_1}.
\end{equation}
Now put $g_2' = g_1\bigr|_{A_2}$. By \eqref{eqn:g1} and (D1), $g_2'\bigr|_{A_1 \cap A_2} \in D_{A_1 \cap A_2}$.
We conclude from (D0) and the continuity of the map $\CCc(A_2,Y) \ni u \mapsto u\bigr|_{A_1 \cap A_2}
\in \CCc(A_1 \cap A_2,Y)$ that there is  a refinement $\GGg_2 \in \cov(Y)$ of $\GGg$ such that
\begin{equation}\label{eqn:G2}
B(g_2',\GGg_2) \subset \{h \in \CCc(A_2,Y)\dd\ h\bigr|_{A_1 \cap A_2} \in D_{A_1 \cap A_2}\}.
\end{equation}
Let $\lambda\dd \Omega \times I \to Y$ be an equiconnecting function (with $\Omega \subset Y \times Y$). Take a cover
$\GGg_2'$ of $Y$ such that for each $G' \in \GGg_2'$,
\begin{equation}\label{eqn:G2'}
G' \times G' \subset \Omega \textup{ and there is $G \in \GGg_2$ for which } \lambda(G' \times G' \times I) \subset G.
\end{equation}
(Notice that this implies that $\GGg_2'$ refines $\GGg_2$.) We infer from \eqref{eqn:g1} and the definition of $\GGg_1$
that $\im g_2' \subset U_0$ and thus $g_2'$ is $U_2$-small in $Y$. Since $U_2$ is a member of $\WwW$, there is
a $\GGg_2'$-close to $g_2'$ map $g_2 \in D_{A_2}$. Now using \eqref{eqn:disjoint} and the assumption that $X$ is normal
take a map $\mu\dd A \to I$ such that
\begin{equation}\label{eqn:mu}
\mu\bigr|_{\overline{A \setminus A_1}} \equiv 1 \quad \textup{and} \quad \mu\bigr|_{\overline{A \setminus A_2}} \equiv 0
\end{equation}
and define $g\dd A \to Y$ as follows: $g\bigr|_{\overline{A \setminus A_1}} = g_2\bigr|_{\overline{A \setminus A_1}}$,
$g\bigr|_{\overline{A \setminus A_2}} = g_1\bigr|_{\overline{A \setminus A_2}}$ and $g(a) = \lambda(g_1(a),g_2(a),\mu(a))$
for $a \in A_1 \cap A_2$. (Note that the last formula makes sense because of \eqref{eqn:G2'} and the fact that $g_2$
is $\GGg_2'$-close to $g_1\bigr|_{A_2}$.) Thanks to \eqref{eqn:mu}, $g$ is a well defined continuous function.
What is more, by (D1) we have
\begin{equation}\label{eqn:A1-A2}
g\bigr|_{\overline{A \setminus A_j}} \in D_{\overline{A \setminus A_j}} \qquad (j=1,2).
\end{equation}
Further, we conclude from \eqref{eqn:G2'} that
\begin{equation}\label{eqn:g2-A2}
g\bigr|_{A_2} \in B(g_2',\GGg_2).
\end{equation}
This, combined with \eqref{eqn:g1}
and the facts that $\GGg_2$ refines $\GGg$ and $\GGg$ is a star refinement of $\VVv$, gives $g \in B(f,\VVv)$.
So, to prove (M2), it suffices to show that $g \in D_A$. But this follows from \eqref{eqn:g2-A2}, \eqref{eqn:G2},
\eqref{eqn:A1-A2} and (D2).\par
We have shown that $\WwW$ is a Michael collection. Therefore, to prove that $D_X$ is dense, it is enough
(thanks to \LEM{michael}) to show that there is $\VVv \in \cov(Y)$ such that $\VVv \subset \WwW$. But LEC spaces
are locally contractible and thus if $\UUu \in \cov(Y)$ is such that
\begin{equation}\label{eqn:smap}
\SsS_Y(X,\UUu) \subset \bar{D}_X,
\end{equation}
then there are $\VVv, \DDd \in \cov(Y)$ such that the family $\{\bar{D}\dd\ D \in \DDd\}$ refines $\UUu$ and each member
of $\VVv$ is contractible in some element of $\DDd$. Now \LEM{NE-LEC}, \eqref{eqn:smap} and (D1) imply that
$\SsS_Y(A,\VVv) \subset \bar{D}_A$ for each $A \in \TTt$, which means that $\VVv \subset \WwW$.\par
Finally, if $A \in \TTt$, then $Y$ is a neighbourhood extensor for $A$ (and $A$ is normal). Hence, by the above
argument, $D_A$ is dense iff it has SMAP. But we have shown that $\VVv \subset \WwW$ for some $\VVv \in \cov(Y)$,
which gives SMAP for every closed subset of $X$.
\end{proof}

\begin{dfn}{normal}
Let $\TTt$ be a closed hereditary class of topological spaces and $Y$ be a paracompact space.
A class $\{D_A\}_{A \in \TTt}$ which satisfies the conditions (D0)--(D3) appearing in the statement of \THM{smap}
is said to be a \textit{normal system over $Y$}. Whenever we deal with normal systems, the underlying class
$\TTt$ is supposed to be closed hereditary.
\end{dfn}

As a simple consequence of \THM{smap} we obtain the following result (we omit its proof).

\begin{pro}{smap2}
Let $X$ and $Y$ be normal spaces such that $Y$ is a hereditary paracompact LEC space and it is a neighbourhood extensor
for $X$. Let $\TTt$ and $\OOo$ be the families of all closed subsets of $X$ and of all open subsets of $Y$,
respectively. Suppose that $\{D_{A,U}\}_{A \in \TTt}^{U \in \OOo}$ is such a collection that for every
$U \in \OOo$ the family $\{D_A = D_{A,U}\}_{A \in \TTt}$ is a normal system over $U$ and
\begin{enumerate}
\item[\upshape(D$*$)] $D_{X,Y} \cap \SsS_Y(X,U) \subset D_{X,U}$.
\end{enumerate}
Then each of the sets $D_{A,U}$ is dense in $\CCc(A,U)$ provided $D_{X,Y}$ has SMAP.
\end{pro}

\begin{rem}{smap}
Under the notation and assumptions of \THM{smap}, the fact that $D_X$ has SMAP is equivalent to the following:
$X$ may be covered by a finite family of closed sets $X_1,\ldots,X_n$ such that $D_{X_j}$ has SMAP for each $j$.
This easily follows from \THM{smap}, (D0)--(D2) and \LEM{NE-LEC}--(B).
\end{rem}

Our next aim is to give equivalent conditions under which every member of a normal system $\{D_A\}_{A \in \TTt}$
over an ANR $Y$ is dense, when $\TTt$ is a rich class of topological spaces (such as compact, metrizable, of weight
no greater than a fixed cardinal, of dimension no greater than a fixed natural number, etc.). This can be done
by a simple adaptation of the concept of Toru\'{n}czyk \cite{tor1}.\par
For a normal system $\{D_A\}_{A \in \TTt}$ over a space $Y$ let us consider the following axioms:
\begin{enumerate}[\upshape(D1)]\setcounter{enumi}{3}
\item If $A \in \TTt$, $B$ is a closed subset of $A$, $f \in \CCc(A,Y)$ and $f\bigr|_B \in D_B$, then there is
   a closed subset $K$ of $A$ such that $f\bigr|_K \in D_K$ and $B \subset \intt_A K$.
\item If $A,B \in \TTt$ and $f \in \CCc(A,B)$, then $D_B f \subset D_A$ (that is, $g \circ f \in D_A$ for each
   $g \in D_B$).
\item If $A, B \in \TTt$ and $A \cong B$, then there is a homeomorphism $h\dd A \to B$ such that $D_B h = D_A$.
\end{enumerate}
Note that (D6) follows from (D5) and it implies that if $A$ and $B$ are two homeomorphic members of $\TTt$, then
$D_A$ is dense [has SMAP] iff $D_B$ is dense [has SMAP]. A normal system satisfying the axiom (D4) is said to be
\textit{strongly normal}. If (D6) [(D5)] is fulfilled, we add the epithet \textit{topological} [\textit{transitive}]
(thus we may talk about topological normal systems, transitive strongly normal systems, etc.). Usually normal systems
are topological.\par
Before we formulate results on topological strongly normal systems, we will establish notation and terminology.\par
Let $\HHh$ be a Hilbert space of dimension $\alpha > 0$ and let $\EEe$ be an orthonormal basis of $\HHh$.
Fix $e \in \EEe$. For a number $n \in \NNN \setminus \{0\}$, let $J_n(\alpha)$ [$K_n(\alpha)$] consists of all nonempty
finite subsets $\sigma$ of $\EEe$ such that $\card (\sigma \cup \{e\}) \leqsl n+1$ [$\card \sigma \leqsl n+1$].
It is clear that $|J_n(\alpha)|_m$ [$|K_n(\alpha)|_m$] is an AR [ANR] of dimension $n$ and of weight
$\max(\alpha,\aleph_0)$ and that every simplicial complex of dimension less than $n$ [no greater than $n$]
which has at most $\alpha-1$ ($=\alpha$ if $\alpha$ is infinite) vertices is isomorphic to a subcomplex of $J_n(\alpha)$
[of $K_n(\alpha)$].\par
Let us agree that $\tau$ is one of the topologies---weak or metric---for simplicial complexes and it is fixed. That is,
whenever in the sequel appears a space of the form $|\KKk|_{\tau}$, where $\KKk$ is a simplicial complex, then $\tau$
always means `$w$' or always means `$m$'. For simplicity, we say that a class $\TTt$ is \textit{corelated} to an ANR $Y$
if every member of $\TTt$ is a normal space for which $Y$ is a neighbourhood extensor and for each $X \in \TTt$ there
is an open cover $\VVv$ of $Y$ (depending on $X$) such that for every $V \in \VVv$:
\begin{enumerate}[\upshape(i)]
\item $\comp(\bar{V}) \leqsl \comp_l(Y)$
\item $\bar{V}$ is FDT provided $Y$ is locally FDT,
\item for any map $f\dd X \to V$ and an open cover $\UUu$ of $Y$ there is a simplicial complex $\KKk$ which witnesses
   \LEM{complex} and such that $\TTt$ contains a space homeomorphic to $|\KKk|_{\tau}$.
\end{enumerate}

From now to the end of the section, we assume that $Y$ is an ANR, $\TTt$ a closed hereditary class corelated to $Y$
and $\Dd = \{D_A\}_{A \in \TTt}$ is a transitive normal system over $Y$. Our purpose is to answer the question of when
\begin{equation*}
\textit{$D_X$ is dense for each space $X \in \TTt$.} \tag{$\star$}
\end{equation*}

The following is left as an exercise.

\begin{lem}{fdc}
\begin{enumerate}[\upshape(A)]
\item \textup{(cf. \cite[Lemma~3.6]{tor1})} Suppose $\Dd$ is strongly normal. Let $K \in \TTt$ be such that
   $K \cong |\KKk|_{\tau}$ where $\KKk$ is a finite dimensional simplicial complex of dimension $n \in \NNN$ which has
   $\alpha > 0$ vertices. If $\TTt$ contains a space $Z \cong I^n \times \alpha$ such that $D_Z$ has SMAP, then $D_K$
   is dense in $\CCc(K,Y)$.
\item \textup{(cf. \cite[part of the proof of Lemma~3.8]{tor1})} If $D_K$ has SMAP for each space $K \in \TTt$
   homeomorphic to a simplicial complex space, then \textup{($\star$)} is fulfilled.
\item \textup{(cf. \cite[proof of Lemma~3.8]{tor1})} Let $K \in \TTt$ be a space homeomorphic to $|\KKk|_{\tau}$
   for some locally finite dimensional simplicial complex $\KKk$ having $\alpha > 0$ vertices. If the class $\TTt$
   contains a space $J(\alpha)$ homeomorphic to $\bigoplus_{n=1}^{\infty} |J_n(\alpha)|_{\tau}$ such that $D_{J(\alpha)}$
   has SMAP, then $D_K$ is dense.
\end{enumerate}
\end{lem}

Let $\comp(\TTt \wedge Y) = \sup\{\min(\comp(X),\comp_l(Y))\dd\ X \in \TTt\}$ and $\dim(\TTt \wedge Y)
= \sup\{\min(\dim(X),\dim(Y))\dd\ X \in \TTt\}$. \LEM{fdc} yields the following

\begin{thm}{dense}
In each of the following cases \textup{($\star$)} is fulfilled.
\begin{enumerate}[\upshape(I)]
\item $\TTt$ consists of compact spaces or $Y$ is locally compact, and for each $m \in \NNN \cap
   [0,\dim(\TTt \wedge Y)]$ there is $n \geqsl m$ such that $I^n \in \TTt$ and $D_{I^n}$ has SMAP.
\item $\Dd$ is strongly normal, $\TTt$ consists of FDT spaces or $Y$ is locally FDT and for each $m \in \NNN \cap
   [0,\dim(\TTt \wedge Y)]$ and positive $\alpha < \comp(\TTt \wedge Y)$ there is $n \geqsl m$ such that
   $I^n \times \alpha \in \TTt$ and $D_{I^n \times \alpha}$ has SMAP.
\item $\TTt$ consists of Lindel\"{o}f spaces or $Y$ is locally separable, $J = \bigoplus_{n=1}^{\infty} I^n \in \TTt$
   and $D_J$ has SMAP.
\item For each positive $\alpha < \comp(\TTt \wedge Y)$, $J(\alpha) = \bigoplus_{n=1}^{\infty} |J_n(\alpha)|_{\tau}
   \in \TTt$ and $D_{J(\alpha)}$ has SMAP.
\end{enumerate}
\end{thm}

Since $Q \setminus \{\textup{point}\} \cong Q \times [0,1)$ (\cite{chapmanbook}), \THM{smap} and \THM{dense}--(III) give

\begin{cor}{separable}
If $X$ is separable, $Q_* = Q \setminus \{\textup{point}\} \in \TTt$ and $D_{Q_*}$ has SMAP,
then $D_Y$ is dense for each $Y \in \TTt$.
\end{cor}

\SECT{Transitive strongly normal systems\\over special ANR's}

In case of special ANR's such as manifolds modelled on infinite-dimensional Hilbert spaces the points (III)--(IV)
of \THM{dense} may be weakened as it will be done in this section.\par
We begin with the following result.

\begin{lem}{loc-fin}
Let $X$ be a paracompact space and $\{D_A\}_{A \in \TTt}$ a normal system over $X$. If $A \in \TTt$,
$A = \bigcup_{s \in S} A_s$ where $\{A_s\}_{s \in S}$ is a discrete family of closed subsets of $A$,
and $g \in \CCc(A,X)$ is such that $g\bigr|_{A_s} \in D_{A_s}$ for each $s \in S$ and the family $\{g(A_s)\}_{s \in S}$
is locally finite in $X$, then $g \in D_A$.
\end{lem}
\begin{proof}
Let $\WwW$ be the collection of all open subsets $U$ of $X$ such that $g\bigr|_{g^{-1}(\bar{U})} \in D_{g^{-1}(\bar{U})}$.
It is easy to show that $\WwW$ is a Michael collection and each point of $X$ has a neighbourhood belonging to $\WwW$.
So, \LEM{michael} finishes the proof.
\end{proof}

For spaces $X$ and $Y$, a discrete collection $\BBb = \{B_s\}_{s \in S}$ of closed subsets of $X$ and a subset $B$ of $X$,
put
\begin{equation}
\LLl(B,Y;\BBb) = \{f \in \CCc(B,Y)\dd\ \{f(B \cap B_s)\}_{s \in S} \textup{ is locally finite in } Y\}.
\end{equation}
It is not difficult to prove that if $Y$ is hereditary paracompact, $\OOo$ and $\TTt$ consist of all open subsets of $Y$
and closed subsets of $X$, respectively, and $D_{A,U} = \LLl(A,U;\BBb)$ for $A \in \TTt$ and $U \in \OOo$, then the axiom
(D$*$) is fulfilled and for each $U \in \OOo$ the system $\{D_A = D_{A,U}\}_{A \in \TTt}$ is strongly normal over $U$.\par
Following Banakh and Zarichnyy \cite{b-z} (cf. \cite[($*$2), p.~253]{tor1}), we say that an ANR $X$ has
the \textit{countable locally finite approximation property} (briefly, $\omega$-LFAP) if for every $\UUu \in \cov(X)$
there is a sequence of maps $\{f_n\dd X \to X\}_{n \in \NNN}$ such that each $f_n$ is $\UUu$-close to $\id_X$
and the family $\{f_n(X)\}_{n \in \NNN}$ is locally finite in $X$. Equivalently (using similar method as those
in \cite{tor1} or in the proof of \LEM{fdc}), $X$ has $\omega$-LFAP iff the family
$\LLl(\oplus_{n \in \NNN} L_n,X;\{L_n\}_{n \in \NNN})$ has SMAP with
\begin{enumerate}[\upshape(LF1)]
\item $\{L_n\}_{n \in \NNN} = \{I^{n+1}\}_{n \in \NNN}$ provided $X$ is locally separable,
\item $\{L_n\}_{n \in \NNN} = \{|J_{n+1}(\alpha)|_{\tau}\}_{n \in \NNN}$ for every infinite $\alpha < \comp_l(X)$,
   otherwise.
\end{enumerate}
Further, $X$ is said to have the \textit{$\kappa$-discrete $m$-cells property} (\cite{b-z},
cf. \cite[($*$1), p.~252]{tor1}) if the family
\begin{equation}\label{eqn:k-SDAP}
\{f \in \CCc(I^m \times \kappa,X)\dd\ \{f(I^m \times \{\beta\})\}_{\beta < \kappa}
\textup{ is discrete in } X\}
\end{equation}
is dense in $\CCc(I^m \times \kappa,X)$. Finally, $X$ has the \textit{strong discrete approximation property}
(briefly, SDAP; \cite{acm}, cf. \cite[Corollary 3.2]{tor1}) if the family
$$
\{f \in \CCc(\oplus_{n \in \NNN} I^{n+1},X)\dd\ \{f(I^{n+1})\}_{n \in \NNN} \textup{ is discrete in } X\}
$$
is dense in $\CCc(\oplus_{n \in \NNN} I^{n+1},X)$. These three concepts were used to characterize Hilbert manifolds
(\cite{tor1,tor2}):
\begin{enumerate}[\upshape(H1)]
\item $X$ is a paracompact manifold modelled on $l^2$ iff $X$ is a locally separable completely metrizable ANR
   which has SDAP,
\item $X$ is a paracompact manifold modelled on a Hilbert space of dimension $\alpha > \aleph_0$ iff $X$ is
   a completely metrizable ANR of local weight $\alpha$ which has $\omega$-LFAP and has $\alpha$-discrete $n$-cells
   property for each $n \in \NNN$.
\end{enumerate}
Note also that SDAP is equivalent to $\omega$-LFAP for locally separable ANR's (\cite{tor1}, \cite{curtis})
and that $\kappa$-discrete $m$-cells property (with infinite $\kappa$) is implied by its `locally finite version'
(\cite[Lemma~4]{banakh}), that is, we may replace the word `discrete' in \eqref{eqn:k-SDAP} by `locally finite'.
Further, we infer from \PRO{smap2} that if $X$ has one of these three above defined properties, every open subset of $X$
has it as well. This explains that the condition \cite[($*$2), p.~253]{tor1} is implied by SDAP (for locally separable
$X$). This also shows that in the definitions of SDAP, $\omega$-LFAP and discrete cells approximation properties we may
replace the word `dense' by `has SMAP'.\par
Banakh \cite{banakh} (see also \cite[Theorem~3.1]{brz}) has proved that a connected ANR of weight $\alpha$ is homeomorphic
to a homotopy dense subset of a Hilbert manifold iff it has $\omega$-LFAP and $\alpha$-discrete $n$-cells property
for each $n$. Recall that a subset $A$ of an ANR $Y$ is \textit{homotopy dense} in $Y$ iff there is a homotopy
$H\dd Y \times I \to Y$ such that $H(y,0) = y$ for each $y \in Y$ and $H(Y \times (0,1]) \subset A$
(\cite{brz}, \cite{banakh}). Banakh's result and the above comments show that

\begin{thm}{loc-dense}
If an ANR $X$ is locally homotopy dense embeddable in Hilbert manifolds and $w(U) = w(X)$ for each nonempty open
subset $U$ of $X$, then $X$ itself is homotopy dense embeddable in a Hilbert manifold.
\end{thm}

Now we shall develop the ideas of the previous section in case of ANR's with $\omega$-LFAP or discrete
cells properties. From now on, we assume that $Y$ is an ANR, $\TTt$ a closed hereditary class corelated to $Y$
and $\{D_A\}_{A \in \TTt}$ is a transitive strongly normal system over $Y$. Additionally, we assume that $\TTt$
contains the spaces $\oplus_{n \in N} I^{n+1}$, and $I^m \times \alpha$ and $J(\alpha) =
\oplus_{n \in \NNN} |J_{n+1}(\alpha)|_{\tau}$ for each positive $m \in \NNN$ and $\alpha < \comp_l(Y)$.

\begin{pro}{LFAP}
In each of the following cases \textup{($\star$)} is fulfilled.
\begin{enumerate}[\upshape(I)]
\item $Y$ has $\omega$-LFAP and $D_{I^n \times \alpha}$ has SMAP for each positive $n \in \NNN$
   and $\alpha < \comp_l(Y)$.
\item $Y$ is locally separable and has SDAP and $D_{I^n}$ has SMAP for each $n \geqsl 1$.
\end{enumerate}
\end{pro}
\begin{proof}
The proofs of both the points are similar and therefore we shall only show (I).
Let $\alpha < \comp_l(Y)$. By \THM{dense}, it suffices to prove that $D_{J(\alpha)}$ is dense.
Let $f \in \CCc(J(\alpha),Y)$ and $\UUu \in \cov(Y)$. Take a sequence $\{u_n\dd Y \to Y\}_{n \geqsl 1}$ such that
$u_m \in B(\id_Y,\UUu)$ for each $m$ and the family $\{u_n(Y)\}_{n \geqsl 1}$ is locally finite in $Y$.
By the assumption of (I) and \LEM{fdc}--(A), $D_{|J_n(\alpha)|_{\tau}}$ is dense for each $n \geqsl 1$. So, there are
maps $g_n \in D_{|J_n(\alpha)|_{\tau}}$ such that $g_n \in B(u_n f\bigr|_{|J_n(\alpha)|_{\tau}},\UUu)$.
Put $g = \bigcup_{n=1}^{\infty} g_n \in \CCc(J(\alpha),Y)$. By \LEM{loc-fin}, $g \in D_{J(\alpha)}$.
What is more, $g \in B(f,\st(\UUu))$ (where $\st(\UUu)$ is the star of $\UUu$), which finishes the proof.
\end{proof}

Analogously one may show that

\begin{pro}{LFAP+}
If $Y$ has $\omega$-LFAP and $\alpha$-discrete $m$-cells property for every positive $m \in \NNN$
and $\alpha < \comp_l(Y)$, and $D_{I^n}$ has SMAP for each $n \geqsl 1$, then \textup{($\star$)} is satisfied.
\end{pro}

\SECT{Strong $Z$-sets}

Following Toru\'{n}czyk \cite{tor,tor0,tor1}, we say that a closed subset $A$ of $X$ is a \textit{$Z$-set} in $X$,
if the set $\CCc(Q,X \setminus A)$ is dense in $\CCc(Q,X)$, or---equivalently---if $\CCc(I^n,X \setminus A)$ is dense
in $\CCc(I^n,X)$ for each $n \geqsl 1$. (If $X$ is an ANR, this definition is equivalent to the original one
by Anderson \cite{anderson}.) Similarly, $A$ is said to be a \textit{strong $Z$-set} in $X$ iff for every
$\UUu \in \cov(X)$ there is a map $u\dd X \to X$ which is $\UUu$-close to $\id_X$ and $A \cap \overline{\im}\,u
= \varempty$ (cf. e.g. \cite{bbmw}, \cite{b-m}, \cite{dij1,dij2}). In other words, $A = \bar{A}$ is a [strong] $Z$-set
in $X$ iff $\SsS_X(D,X \setminus A)$ is dense in $\CCc(D,X)$ where $D = Q$ [$D = X$]. Not every $Z$-set in an ANR
is a strong $Z$-set (\cite[Key example, p.~56]{bbmw}). However, combined results of Henderson \cite{henderson}
and Banakh \cite{banakh} show that every $Z$-set in an ANR $X$ having $\omega$-LFAP and $w(X)$-discrete $n$-cells property
for each $n$ is strong. We shall obtain this result independently of the theorems of Henderson and Banakh.\par
The following result is easy to prove.

\begin{lem}{strong}
If $X$ is hereditary paracompact and $A$ its closed subset, $\TTt$ is a class of all topological spaces
and for $Y \in \TTt$ and an open subset $U$ of $X$, $D_{W,U} = \SsS_U(Y,U \setminus A)$, then $\{D_{Y,U}\}_{Y \in \TTt}$
is a transitive strongly normal system over $U$ and \textup{(D$*$)} is fulfilled.
\end{lem}

As an immediate consequence of the above lemma and the results of the previous sections, we obtain

\begin{pro}{strong}
Let $X$ be an ANR and $A$ its closed subset.
\begin{enumerate}[\upshape(Z1)]
\item \textup{(\cite{b-m}, \cite[Proposition~1.4.1]{brz})} If $X$ is locally compact, then every $Z$-set in $X$
   is a strong $Z$-set.
\item \textup{(cf. \cite[Exercise 1.4.10]{brz})} If $X$ is locally separable, then $A$ is a strong $Z$-set in $X$
   \iaoi{} $\SsS_X(\oplus_{n \in \NNN} I^{n+1},X \setminus A)$ has SMAP,
   iff $\SsS_X(Q \setminus \{\textup{point}\},X \setminus A)$ has SMAP.
\item If $X$ is locally FDT, then $A$ is a strong $Z$-set iff $\SsS_X(I^n \times \alpha,X \setminus A)$ has SMAP
   for each positive $n \in \NNN$ no greater than $\dim X$ and each $\alpha < \comp_l(X)$.
\item $A$ is a strong $Z$-set in $X$ iff $\SsS_X(\oplus_{n \in \NNN} |J_{n+1}(\alpha)|_{\tau},X \setminus A)$ has
   SMAP for each positive $\alpha < \comp_l(X)$.
\item \textup{(\cite[Lemma~1.3]{b-m})} If $A$ is a strong $Z$-set in $X$, then $U \cap A$ is a strong
   $Z$-set in $U$ for each open subset $U$ of $X$.
\item \textup{(for separable $X$ see \cite[Corollary~1.5]{b-m})} If every point of $A$ has an open neighbourhood $U$
   in $X$ such that $A \cap U$ is a strong $Z$-set in $U$, then $A$ is a strong $Z$-set in $X$.
\item \textup{(cf. \cite[Proposition~1.7]{b-m} or \cite[Proposition 1.4.3]{brz})} If $X$ is locally separable
   and has SDAP or if $X$ has $\omega$-LFAP and $\alpha$-discrete $m$-cells property for each positive $m \in \NNN$
   and $\alpha < \comp_l(X)$, then every $Z$-set in $X$ is strong.
\end{enumerate}
\end{pro}

The counterpart of (Z6) for $Z$-sets in its whole generality was first proved by Eells and Kuiper \cite{e-k}.
Their proof is based on the theorem of Whitehead \cite{whitehead} on weak homotopy equivalences and they worked
with Anderson's \cite{anderson} definition of a $Z$-set. Here we gave an alternative proof of (Z6) for `Toru\'{n}czyk's
$Z$-sets'.\par
The property (Z5), by a simple use of SMAP, may be strengthened as follows: if $K$ is a strong $Z$-set in a metrizable
space $Y$ and $X$ is an open subset of $Y$ such that $X$ is an ANR, then $K \cap X$ is a strong $Z$-set in $X$.\par
Following \cite{brz}, we say that an ANR $X$ has the \textit{strong $Z$-set property} iff every $Z$-set in $X$ is strong.
(Z5) and (Z6) yield that the strong $Z$-set property is open hereditary (i.e. every open subset of an ANR $X$ has
the strong $Z$-set property provided so has $X$) and is local (if each point of $X$ has an open neighbourhood
with the strong $Z$-set property, then $X$ has it as well). This facts will be used in the next section.\par
Bestvina and Mogilski \cite{b-m} have proved that a $Z$-set being a strong $\sigma$-$Z$-set in a separable ANR is itself
a strong $Z$-set. (This property were used by them to prove (Z6) for separable $X$.) We do not know whether the assumption
of separability of the ANR in this statement may be omitted. The lack of such a property will force us in the next section
to assume that ANR's have the strong $Z$-set property.

\SECT{Absorbing sets}

In this section we generalize most important results of \cite{b-m} to nonseparable case. All undefined symbols
and notions have the same meaning as in \cite{b-m} (after deleting all assumptions dealing with separability).
In particular, a \textit{$\CcC$-absorbing set} is any space $X$ such that: $X$ is strongly $\CcC$-universal and homotopy
dense embeddable into a Hilbert manifold, $X \in \CcC_{\sigma}$ and $X$ is a \textit{$\sigma$-$Z$-space}
(i.e. $X = \bigcup_{n=1}^{\infty} X_n$ for some sequence $(X_n)_{n=1}^{\infty}$ of $Z$-sets in $X$).\par
The proofs presented in \cite{b-m} (see also \cite{brz}) of the following results (we quote only the most important ones)
work also in nonseparable case (Chapman's generalization \cite{chapman} of Anderson's-McCharen's theorem \cite{and-mcch}
is needed).

\begin{thm}{several}
Let $\CcC$ be a topological, closed hereditary and additive class of metrizable spaces and let $X$ be a metrizable space.
\begin{enumerate}[\upshape(SU1)]
\item \textup{(\cite[Proposition 2.1]{b-m})} If $X$ is strongly $\CcC$-universal, so is every its open subset.
\item \textup{(\cite[Proposition 2.2]{b-m})} If $\CcC$ is also open hereditary, $X$ is an ANR having the strong $Z$-set
   property and for each $Z$-set $K$ in $X$ the space $X \setminus K$ is $\CcC$-universal, then $X$ is strongly
   $\CcC$-universal.
\item \textup{(\cite[Proposition 2.7]{b-m}; for general proof see \cite[Proposition~1.5.1]{brz})} If $X$ is an ANR
   having an open cover consisting of strongly $\CcC$-universal sets, then $X$ itself is strongly $\CcC$-universal.
\item \textup{(\cite[Proposition 2.6]{b-m} or \cite[Theorem 1.5.11]{brz})} If $X$ is strongly $\CcC$-universal
   and $Y$ is an ANR such that $X \times Y$ has the strong $Z$-set property, then $X \times Y$ is strongly
   $\CcC$-universal as well.
\item \textup{(\cite[Theorem 3.1]{b-m} or \cite[Theorem 1.6.3]{brz})} Two $\CcC$-absorbing sets are homeomorphic
   iff they have the same homotopy type.
\item \textup{(\cite[Theorem 3.2]{b-m}; $Z$-set Unknotting Theorem)} If $X$ is $\CcC$-ab\-sorbing, then every
   homeomorphism between two $Z$-sets in $X$ which is $\UUu$-homotopic to the inclusion map for some $\UUu \in \cov(X)$
   is extendable to a homeomorphism of the whole space $\st(\UUu)$-close to the identity map on $X$.
\item \textup{(\cite[Proposition 2.5 and Corollaries 5.4 and 5.5]{b-m})} If $X$ is an AR having more than one point,
   then the weak product $$W(X,*) = \{(x_n)_{n=1}^{\infty} \in X^{\omega}\dd\ x_n = * \text{ for almost all } n\}$$
   (where $* \in X$ is a fixed `basepoint') is $\DdD$-absorbing where $\DdD$ is the class of all spaces admitting closed
   embeddings into $W(X,*)$. The class $\DdD$ is topological, closed hereditary and additive. If $Y$ is also an AR,
   then $W(X,*)$ and $W(Y,*)$ are homeomorphic iff $X$ embeds as a closed subset of $W(Y,*)$ and $Y$ embeds as a closed
   subset of $W(X,*)$.
\end{enumerate}
\end{thm}

Bestvina and Mogilski have also proved that if a separable ANR $X$ is a $\sigma$-$Z$-space, has the strong $Z$-set
property and is strongly $\CcC$-universal for a topological closed hereditary additive class $\CcC$, then it is also
strongly $\CcC_{\sigma}$-universal (see \cite[Proposition 2.3]{b-m}, the proof uses in its final part the second axiom
of countability). We do not know whether the assumption of separability of $X$ may be omitted in this. However, below
we prove its counterpart for absorbing sets (which, alternatively, may be used to prove the last claim of (SU7)).
This result will be applied in the next section.\par
One of the most important results on separable absorbing sets, beside (SU5), states that in the definition
of an absorbing set one may omit the assumption of homotopy dense embeddability into a Hilbert manifold---this is
a consequence of \cite[Lemma 1.9]{b-m} and Banakh's theorem \cite{banakh}. It turns out that this is true also
for nonseparable absorbing sets, which shows

\begin{pro}{sigma}
Let $\CcC$ be a topological, closed hereditary and additive class and $X$ be an ANR which is strongly $\CcC$-universal,
has the strong $Z$-set property and is a $\sigma$-$Z$-space. If $X \in \CcC_{\sigma}$, then $X$ is
$\CcC_{\sigma}$-absorbing.
\end{pro}
\begin{proof}
For simplicity, put $\kappa = w(X)$. The proof of \cite[Proposition 2.3]{b-m} shows that the set of closed embeddings
is dense in $\CCc(P,X)$ for any $P \in \CcC_{\sigma}$. Since $X \in \CcC_{\sigma}$, we get that $X \times \NNN \in
\CcC_{\sigma}$ as well and thus the natural projection $X \times \NNN \to X$ is approximable by closed embeddings.
This easily implies that $X$ has $\omega$-LFAP. What is more, by \LEM{omega}, $X$ contains a discrete subset
of cardinality $\kappa$, say $A$. Now if $\{U_a\}_{a \in A}$ is a discrete family of open subsets of $X$ such that
$a \in U_a$ for $a \in A$, there is a family $\{h_a\dd X \to X\}_{a \in A}$ of closed embeddings such that $\im h_a
\subset U_a$. This shows that $X$ contains a closed subset homeomorphic to $X \times A$
(namely $\bigcup_{a \in A} \im h_a$) and therefore also the natural projection $X \times A \to X$ is approximable
by closed embeddings. Hence $X$ satisfies the $\kappa$-discrete $m$-cells property for each $m$. So, thanks to Banakh's
theorem \cite{banakh}, $X$ is homotopy dense embeddable into a Hilbert manifold. Now we shall check that $X$
is $\CcC_{\sigma}$-universal. To prove this, it is enough to show that $Z$-embeddings form a dense subset
of $\CCc(X,X)$.\par
Let $\{u_{\beta}\}_{\beta < \kappa}$ be a dense subset of $\CCc(Q,X)$. Further, let $$L = \oplus_{\beta < \kappa}
\im u_{\beta}$$ be the topological disjoint union of the images of the maps $u_{\beta}$, $X'$ be the topological
disjoint union of $X$ and $L$ and let $v\dd L \to X$ be given by $v(x) = x$. By the above argument, both the spaces
$L$ and $X'$ are members of $\CcC_{\sigma}$. Fix a map $f\dd X \to X$ and a cover $\UUu_0 \in \cov(X)$ and put
$f_0 = f \cup v\dd X' \to X$. Write $X = \bigcup_{n=1}^{\infty} X_n$ where each $X_n$ is a (strong) $Z$-set in $X$,
$X_n \in \CcC$ and $X_n \subset X_{n+1}$, and let $d$ denote the metric of $X$. Now arguing similarly as in the
proof of \cite[Proposition~2.3]{b-m}, construct sequences of closed embeddings $\{f_n\dd X' \to X\}_{n=1}^{\infty}$
and covers $\{\UUu_n\}_{n=1}^{\infty}$ of $X$ such that for each $j \geqsl 1$,
\begin{enumerate}[\upshape(E1)]
\item $f_j\bigr|_{X_j}$ is a $Z$-embedding,
\item $f_j\bigr|_{X_{j-1}} = f_{j-1}\bigr|_{X_{j-1}}$ (with $X_0 = \varempty$),
\item $f_j$ is $\UUu_j$-close to $f_{j-1}\bigr|_X \cup v\dd X' \to X$,
\item $\UUu_j$ is a star refinement of $\UUu_{j-1}$, $\mesh_d \UUu_j < 2^{-j}$ and if $g\dd X \to X$ is $\UUu_j$-close
   to $f_{j-1}\bigr|_X$, then $\overline{\im}\,g \cap \bigcup_{k=1}^{j-1} f_k(L) = \varempty$
   (where $\bigcup_{k=1}^0 = \varempty$),
\end{enumerate}
and the formula $h(x) = \lim_{n\to\infty} f_n(x)$ well defines a closed embedding $h\dd X \to X$. Then $h$
is $\UUu_0$-close to $f$ and $h(X) \cap \bigcup_{n=1}^{\infty} f_n(L) = \varempty$. But thanks to (E3) and (E4),
the family $$\{f_n\bigr|_{\im u_{\beta}} \circ u_{\beta}\dd\ n \geqsl 1,\ \beta < \kappa\}$$ is dense in $\CCc(Q,X)$
and thus $h(X)$ is a $Z$-set in $X$.\par
To finish the proof, it suffices to apply the above argument for each open subset $U$ of $X$ (instead of $X$)
and use (SU2).
\end{proof}

The next two results will be used in the next section.

\begin{cor}{F-sigma}
Under the assumptions of \PRO{sigma}, every $\FFf_{\sigma}$ subset of $X$ is closed embeddable in $X$.
\end{cor}

\begin{pro}{cl-emb}
Let $X$ be a noncompact AR which contains a closed subset homeomorphic to $X \times X$. Let $Y$ be a metrizable space
such that $w(Y) \leqsl w(X)$. Each of the following two conditions is sufficient for the closed embeddability of $Y$
in $X$.
\begin{enumerate}[\upshape(A)]
\item Every point of $Y$ has a neighbourhood (not necessarily open or closed) which is closed embeddable in $X$.
\item $Y$ may be covered by a locally finite collection of its closed subsets each of which is closed embeddable in $X$.
\end{enumerate}
\end{pro}
\begin{proof}
First assume that $Y$ is the union of its two closed subsets $Y_1$ and $Y_2$ which are closed embeddable in $X$.
Since $X$ is an AR, there is a map $u_j\dd Y \to X$ such that $u_j\bigr|_{Y_j}$ is a closed embedding. Further, take
a map $\lambda\dd Y \to I$ which is positive on $Y \setminus Y_1$ and negative on $Y \setminus Y_2$ (such a map exists
because $Y_1 \cap Y_2$ is closed and $\GGg_{\delta}$ in both the spaces $Y_1$ and $Y_2$). Now put $u\dd Y \ni y \mapsto
(u_1(y),u_2(y),\lambda(y)) \in X \times X \times I$. One easily checks that $u$ is a closed embedding and that there is
a closed embedding of $X \times X \times I$ into $X$ (because of the facts that $X \times X$ is closed embeddable in $X$
and $X$ contains an arc).\par
Now to prove the sufficiency of (A), let $\WwW$ be the family of all open subsets $U$ of $Y$ such that $\bar{U}$ is closed
embeddable in $X$. Thanks to \LEM{michael}, it is enough to prove that $\WwW$ is a Michael collection. The property (M1)
is immediate, (M2) follows from the first part of the proof, and to see (M3), use \LEM{omega} and the fact that $X$
contains a closed copy of $X \times X$ (and therefore also a closed copy of $X \times w(X)$).\par
The sufficiency of (B) is implied by the one of (A) and the first part of the proof (with simple induction argument).
\end{proof}

Let us call a class $\CcC$ \textit{product} if $C_1 \times C_2 \in \CcC$ for each $C_1$, $C_2 \in \CcC$.\par
For a metrizable space $X$, let $\FfF(X)$ be the class of all metrizable spaces which are closed embeddable in $X^n
\times I^n$ for some $n \in \NNN$. It is easily seen that $\FfF(X)$ is topological, closed hereditary and product
and that $\FfF(X)$ coincides with the class of metrizable spaces admitting closed embeddings in $X^n$
for some $n \in \NNN$ provided $X$ contains an arc. What is more, the first paragraph of the proof of \PRO{cl-emb} shows
that $\FfF(X)$ is additive for an AR $X$. This facts will be used later.

\begin{rem}{final}
In Final Remarks (on page 425) of \cite{d-mo} Dobrowolski and Mogilski give two classical examples of nonseparable
absorbing sets, namely $l^2_f(A)$ and $\Sigma l^2(A)$ (where $A$ is an uncountable set). They also mention (SU5)
in its full generality (compare \EXM{fd}).
\end{rem}

\SECT{Rigid embeddings}

The main aim of this section is to prove that the weak product $W(X,*)$ (defined in (SU7)) for an arbitrary AR $X$
is homeomorphic to some pre-Hilbert space $E$ with $E \cong \Sigma E$, where
$$\Sigma E = \{(x_n)_{n=1}^{\infty} \in E^{\omega}\dd\ x_n = 0 \textup{ for almost all } n\}.$$
To do this, we shall introduce and investigate certain types of embeddings. But first we establish notation.
For $t = (t_1,\ldots,t_n) \in \RRR^n$ let $|t|_{\infty} = \max_j |t_j|$ and $|t|_p = (\sum_{j=1}^n |t_j|^p)^{\frac1p}$.
Let $E$ be a real vector space. For $s = (s_1,\ldots,s_n) \in \RRR^n$ and $v = (v_1,\ldots,v_n) \in E^n$, put
$$
s \bullet v = \sum_{j=1}^n s_j v_j.
$$
More generally, if $h\dd D \to E$ is any function (with an arbitrary domain $D$) and $x = (x_1,\ldots,x_n) \in D^n$,
then $s \bullet h(x)$ denotes the vector $\sum_{j=1}^n s_j h(x_j)$. If $y = (y_1,\ldots,y_n) \in D^n$ and $\sigma$
is a permutation of the set $\{1,\ldots,n\}$, $y_{\sigma}$ stands for the $n$-tuple $(y_{\sigma(1)},\ldots,y_{\sigma(n)})$.
Finally, for a metric space $(X,d)$ and a natural number $n \geqsl 2$, $\delta_n\dd X^n \to [0,+\infty)$ is a function
defined by
$$
\delta_n(x) = \min \{d(x_j,x_k)\dd\ j \neq k\}, \qquad x = (x_1,\ldots,x_n) \in X^n.
$$
Additionally, we put $\delta_1 \equiv 1\dd X \to [0,+\infty)$. Now we are ready to put the following

\begin{dfn}{rigid}
Let $(X,d)$ be a metric space. A map $h\dd X \to E$ where $(E,\|\cdot\|)$ is a normed space is said to be
\textit{weakly separately rigid} (with respect to $d$) if for each natural number $n \geqsl 1$ and reals $r > 0$
and $M \geqsl 1$ there is a constant $C = C(r,n,M)$ such that $\|t \bullet h(x)\| \geqsl C$ whenever $t = (t_1,\ldots,t_n)
\in \RRR^n$ and $x \in X^n$ are such that $\min_j |t_j| \geqsl \frac1M$, $|t|_{\infty} \leqsl M$ and $\delta_n(x) \geqsl
r$. If the constant $C$ can be chosen independently of $M$ (that is, $C = C(r,n)$) for each $n$ and $r$, the map $h$
is said to be \textit{separately rigid} and if $C$ depends only on $r$, $h$ is called \textit{almost rigid}. Finally,
$h$ is \textit{rigid} if for each $r > 0$ there is a constant $C = C(r)$ such that $\|t \bullet h(x)\| \geqsl
C |t|_{\infty}$ for each $t \in \RRR^n$ and $x \in X^n$ with $\delta_n(x) \geqsl r$, and any $n \geqsl 1$.\par
A subset $A$ of a normed space $E$ is said to be \textit{rigid}, \textit{almost rigid}, etc. if the inclusion map
$A \inj E$ is (respectively) rigid, almost rigid, etc.
\end{dfn}

It is easy to see that a weakly separately rigid map is injective and its image is linearly independent. Note also
that if a uniformly continuous map is rigid, almost rigid, separately rigid or weakly separately rigid, then so is its
image. In the sequel we shall show that every weakly separately rigid map is an embedding whose image is closed in its
own linear span. To see this, we need the following

\begin{lem}{rigid}
Let $h\dd X \to E$ be a weakly separately rigid map of a metric space $(X,d)$ into a normed space $(E,\|\cdot\|)$
such that the linear span of $\im h$ coincides with $E$. Let $p \geqsl 1$, $r > 0$, $M \geqsl 1$,
$t^{(n)} = (t^{(n)}_1,\ldots,t^{(n)}_p) \in [-M,-\frac1M]^p \cup [\frac1M,M]^p \subset \RRR^p$, $x^{(n)} =
(x^{(n)}_1,\ldots,x^{(n)}_p) \in X^p$ be such that $\delta_p(x^{(n)}) \geqsl r$ and the sequence $(t^{(n)} \bullet
h(x^{(n)}))_{n=1}^{\infty}$ converges in $E$. Then there is a sequence $(\sigma_n)_{n=1}^{\infty}$ of permutations
of the set $\{1,\ldots,p\}$ such that both the sequences $(t^{(n)}_{\sigma_n})_{n=1}^{\infty}$
and $(x^{(n)}_{\sigma_n})_{n=1}^{\infty}$ converge in $\RRR^p$ and $X^p$ respectively.
\end{lem}
\begin{proof}
Let us agree that $\sum_{j=1}^0 = 0$. It suffices to prove that (under the assumptions of the lemma on $t^{(n)}$
and $x^{(n)}$):
\begin{enumerate}
\item[($\bullet$)] If $\sum_{j=1}^p t^{(n)}_j h(x^{(n)}_j) \to \sum_{k=1}^q t_k h(x_k)$ for some $q \geqsl 0$, nonzero
   reals $t_1,\ldots,t_q$ and distinct points $x_1,\ldots,x_q$ of $X$, then $q = p$ and for some sequence
   $(\sigma_n)_{n \geqsl 1}$ of permutations one has $t^{(n)}_{\sigma_n(j)} \to t_j$ and $x^{(n)}_{\sigma_n(j)} \to x_j$
   for $j=1,\ldots,p$.
\end{enumerate}
We shall show this by induction on $p$, starting with $p = 0$.
For $p = 0$ we only need to note that the image of $h$ is linearly independent and thus $q = 0$. Now assume that we have
proved ($\bullet$) for $p-1$. The proof of ($\bullet$) for $p$ is divided into three steps.\par
\begin{enumerate}
\item[I.] If $x^{(n)}_1 \to x_1$, then $t^{(n)}_1 \to t_1$.
\end{enumerate}
Proof of I. We argue by contradiction. Suppose that $(t^{(n)}_1)_{n \geqsl 1}$ does not converge to $t_1$. Passing
to a subsequence, we may assume that $t^{(n)}_1 \to s \neq t_1$. But then $t^{(n)}_1 h(x^{(n)}_1) \to s h(x_1)$
and hence $\sum_{j=2}^p t^{(n)}_j h(x^{(n)}_j) \to (t_1 - s) h(x_1) + \sum_{k=2}^q t_k h(x_k)$. Now by the induction
hypothesis we have $q = p-1$ and $x^{(n)}_{l_n} \to x_1$ for some sequence of $l_n \in \{2,\ldots,p\}$. This yields
$r \leqsl \delta_p(x^{(n)}) \leqsl d(x^{(n)}_1,x^{(n)}_{l_n}) \leqsl d(x^{(n)}_1,x_1) + d(x_1,x^{(n)}_{l_n}) \to 0$,
which is impossible.
\begin{enumerate}
\item[II.] There are a sequence of $l_n \in \{1,\ldots,p\}$ and $k \in \{1,\ldots,q\}$ such that
   $x^{(n)}_{l_n} \to x_k$.
\end{enumerate}
Proof of II. Observe that $q > 0$, because $h$ is weakly separately rigid. Put $\alpha_n = \min \{d(x^{(n)}_j,x_k)\dd\
j \in \{1,\ldots,p\},\ k \in \{1,\ldots,q\}\}$. It is enough to show that $\lim_{n\to\infty} \alpha_n = 0$. Suppose
that the latter convergence does not hold. Passing to a subsequence, we may assume that $\alpha_n \geqsl c$ for all $n$
and some positive constant $c$. This implies that there is $\epsi > 0$ such that $\delta_{p+q}(y^{(n)}) \geqsl \epsi$
for each $n$ where $y^{(n)} = (x^{(n)}_1,\ldots,x^{(n)}_p,x_1,\ldots,x_q) \in X^{p+q}$. Moreover, there is a constant
$A \geqsl 1$ such that $s^{(n)} \in [-A,-\frac1A]^{p+q} \cup [\frac1A,A]^{p+q} \subset \RRR^{p+q}$, where $s^{(n)} =
(t^{(n)}_1,\ldots,t^{(n)}_p,-t_1,\ldots,-t_q)$. We have $s^{(n)} \bullet h(y^{(n)}) \to 0$ which denies the weak separate
rigidity of $h$.
\begin{enumerate}
\item[III.] There is a sequence of $\sigma_n(1) \in \{1,\ldots,p\}$ such that $x^{(n)}_{\sigma_n(1)}
   \to x_1$.
\end{enumerate}
Proof of III. As in the proof of II, it suffices to show that the sequence of $\gamma_n = \min_j d(x^{(n)}_j,x_1)$
tends to $0$. Again, we argue by contradiction. Passing to a subsequence, applying II and after renumerating of $x^{(n)}$
(which depends on $n$), we may assume that
\begin{equation}\label{eqn:aux}
\gamma_n \geqsl \epsi
\end{equation}
for all $n$ and a positive constant $\epsi$, and $x^{(n)}_1 \to x_c$ for some $c \in \{2,\ldots,q\}$. By I,
$t^{(n)}_1 \to t_c$ and thus $\sum_{j=2}^p t^{(n)}_j h(x^{(n)}_j) \to \sum_{k \neq c} t_k h(x_k)$. So, we infer
from the induction hypothesis that $q = p$ and $x^{(n)}_{j_n} \to x_1$ for some $j_n \in \{2,\ldots,p\}$,
which contradicts \eqref{eqn:aux}.\par
Now ($\bullet$) follows from III, I and the induction hypothesis.
\end{proof}

Applying the above lemma with $p = r = M = 1$ we obtain:

\begin{cor}{embed}
Every weakly separately rigid map $h\dd X \to E$ such that $\lin h(X) = E$ is a closed embedding.
\end{cor}

\begin{pro}{c-sigma}
Let $(X,d)$ be a nonempty metric space and let $\CcC = \FfF(X)$. If $h\dd X \to E$ is a weakly separately rigid map
into a normed space $(E,\|\cdot\|)$ such that $\lin h(X) = E$, then $X \in \FfF(\Sigma E)$
and $\Sigma E \in \CcC_{\sigma}$.
\end{pro}
\begin{proof}
Since the class $\CcC$ is product, we only need to check that $E \in \CcC_{\sigma}$ ($X \in \FfF(\Sigma E)$ thanks
to \COR{embed}). For (positive) natural numbers $p$ and $q$ let $G_{p,q}$ be the set of all vectors of the form
$t \bullet h(x)$ with $t \in [-q,-\frac1q]^p \cup [\frac1q,q]^p$ and $x \in X^p$ such that $\delta_p(x) \geqsl \frac1q$.
\LEM{rigid} shows that $G_{p,q}$ is closed in $E$. What is more, it is easily seen that $E = \{0\} \cup
\bigcup_{p,q \geqsl 1} G_{p,q}$. So, it suffices to show that each of $G_{p,q}$'s is a member of $\CcC_{\sigma}$.
Fix $p$ and $q$ and put $\Delta = \{(t_1,\ldots,t_p) \in \RRR^p\dd\ \frac1q \leqsl |t_j| \leqsl q\}$,
$\Gamma = \{x \in X^p\dd\ \delta_p(x) \geqsl \frac1q\}$ and $\Phi\dd \Delta \times \Gamma \ni (t,x) \mapsto t \bullet x
\in G_{p,q}$. Observe that $\Delta \times \Gamma \in \CcC$ and $\Phi$ is a continuous surjection. What is more,
$\Phi$---as a map of $\Delta \times \Gamma$ into $G_{p,q}$---is open thanks to \LEM{rigid}, and $\card \Phi^{-1}(\{v\})
= p!$ for each $v \in G_{p,q}$ (because $\Phi(t,x) = \Phi(t',x')$ iff $t' = t_{\sigma}$ and $x' = x_{\sigma}$
for a unique permutation $\sigma$ of $\{1,\ldots,p\}$). This implies that $\Phi$ is a covering and therefore it is
a local homeomorphism. Take a small enough open cover $\{U_s\}_{s \in S}$ of $G_{p,q}$ and a corresponding family
$\{V_s\}_{s \in S}$ of relatively open subsets of $\Delta \times \Gamma$ such that $\Phi$ restricted to each $V_s$
is a homeomorphism of $V_s$ onto $U_s$. Next find an open cover $\WWw = \bigcup_{n=1}^{\infty} \WWw_n$ of $G_{p,q}$
such that $\WWw_n = \{W_{s,n}\}_{s \in S}$ is discrete in $G_{p,q}$ and $W_{s,n} \subset U_s$ for each $s \in S$
and $n \geqsl 1$. Then the set $D_n = \bigcup_{s \in S} W_{s,n}$ is open in $G_{p,q}$ and is homeomorphic
to $\bigcup_{s \in S} (\Phi^{-1}(W_{s,n}) \cap V_s)$. We conclude from this that $D_n \in \CcC_{\sigma}$ and thus also
$G_{p,q} \in \CcC_{\sigma}$ (since $G_{p,q} = \bigcup_{n=1}^{\infty} D_n$ and each $D_n$ is $\FFf_{\sigma}$ in $G_{p,q}$).
\end{proof}

We are now ready to prove

\begin{thm}{rigid}
If $(X,d)$ is an absolute retract having more than one point and $h\dd X \to E$ is a weakly separately rigid map
into a normed space $(E,\|\cdot\|)$ such that $\lin h(X) = E$, then $W(X,*) \cong \Sigma E$.
\end{thm}
\begin{proof}
By (SU7) and \COR{embed}, it is enough to show that $E$ is closed embeddable in $W(X,*)$. But this follows
from \PRO{c-sigma}, (SU7) and \PRO{sigma}.
\end{proof}

Now we shall give an example of rigid embeddings.

\begin{exm}{H}
The following construction is due to Bessaga and Pe\l{}\-czy\'{n}\-ski \cite{b-p2} (or \cite[Proposition VI.7.1]{b-p}).
Let $(X,d)$ be a nonempty metric space. For each $n \geqsl 1$ take a locally finite partition of unity
$\{f_{\lambda}\}_{\lambda \in \Lambda_n}$ such that
\begin{equation}\label{eqn:pou}
f_{\lambda}(x) \cdot f_{\lambda}(y) = 0 \textup{ whenever $\lambda \in \Lambda_n$ and } d(x,y) \geqsl 2^{-n}
\end{equation}
and put $g_{\lambda} = 2^{-n} f_{\lambda}$ (for $\lambda \in \Lambda_n$). Assuming that the sets
$\Lambda_1,\Lambda_2,\ldots$ of indices are pairwise disjoint, put $\Lambda = \bigcup_{n=1}^{\infty} \Lambda_n$ and define
$h\dd X \to l^2(\Lambda) = \{g\dd \Lambda \to \RRR|\ \|g\|_2^2 = \sum_{\lambda \in \Lambda} g(\lambda)^2 < +\infty\}$
by $(h(x))(\lambda) = \sqrt{g_{\lambda}(x)}$. Then $h$ is continuous and $\im h$ is contained in the unit sphere
of $l^2(\Lambda)$. What is more, $h$ is rigid: if $p \geqsl 1$, $t = (t_1,\ldots,t_p) \in \RRR^p$, $x = (x_1,\ldots,x_p)
\in X^p$ and $\delta_p(x) \geqsl 2^{-n}$, then $\|t \bullet h(x)\|_2^2 \geqsl \sum_{\lambda\in\Lambda_n} 2^{-n}
(\sum_{j=1}^p t_j \sqrt{f_{\lambda}(x_j)})^2$ and \eqref{eqn:pou} gives
$$
\|t \bullet h(x)\|_2^2 \geqsl 2^{-n} \sum_{\lambda\in\Lambda_n} \sum_{j=1}^p t_j^2 f_{\lambda}(x_j)
= 2^{-n} \sum_{j=1}^p t_j^2 \sum_{\lambda\in\Lambda_n} f_{\lambda}(x) = 2^{-n} |t|_2^2.
$$
This example shows that every metric space admits a rigid embedding into a Hilbert space.
\end{exm}

Now we have

\begin{thm}{pre-H}
For an arbitrary AR $X$, the weak product $W(X,*)$ is homeomorphic to some pre-Hilbert space $E$ such that
$E \cong \Sigma E$.
\end{thm}
\begin{proof}
Clearly, we may assume that $X$ has more than one point. By \EXM{H} and \THM{rigid}, there is a pre-Hilbert space $F$
with an inner product $\scalarr_F$ such that $W(X,*) \cong \Sigma F$. Let $E = \Sigma_{l^2} F$, i.e. $E$ is the set
$\Sigma F$ with the norm $\|(x_n)_{n=1}^{\infty}\|_E = \sqrt{\sum_{n=1}^{\infty} \scalar{x_n}{x_n}_F}$. There is a natural
inner product in the space $E$ inducing the norm $\|\cdot\|_E$ and, e.g. by \cite[Corollary~1.9]{tor3},
$E \cong \Sigma F$. So, $E \cong \Sigma E \cong W(X,*)$.
\end{proof}

It is easily seen by Toru\'{n}czyk's characterization theorem for Hilbert manifolds \cite{tor1,tor2} that a connected
metrizable space $X$ is a Hilbert manifold iff $X$ is a noncompact completely metrizable ANR such that the natural
projection $X \times X \to X$ is approximable by closed embeddings. In similar spirit, with use of \THM{pre-H}, we now
characterize manifolds modelled on pre-Hilbert spaces $E$ with $E \cong \Sigma E$.

\begin{pro}{charact}
Let $X$ be a connected nonempty metrizable space. Let $\Omega = X$ if $X$ is contractible and otherwise
let $\Omega$ be the topological open or closed cone over $X$. \TFCAE
\begin{enumerate}[\upshape(i)]
\item $X$ is homeomorphic to (an open subset of) some nonzero pre-Hilbert space $E$ with $E \cong \Sigma E$,
\item $X$ is a $\CcC$-absorbing AR (ANR) for some topological, closed hereditary, additive and product class $\CcC$,
\item $X$ is an AR (ANR) and a $\sigma$-$Z$-space such that for every $Z$-set $K$ in $X$ the natural projection
   $(X \setminus K) \times \Omega \to X \setminus K$ is a near-homeomorphism, i.e. it is approximable by homeomorphisms.
\end{enumerate}
\end{pro}
\begin{proof}
To prove that (iii) follows from (i) it suffices to apply a variation (cf. \cite[Theorem 5]{h-w}) of Schori's
theorem \cite{schori} (note that both the open and closed cones over $E$ are homeomorphic to $E$ for spaces $E$
as in (i)---see e.g. \cite{henderson2}---and thus if $X$ is an open subset of $E$, then both the open and closed cones
over $X$ are factors of $E$, thanks to \cite{tor4,tor3} or \cite[Corollary~5.4]{b-m}).\par
To see that (i) is implied by (ii), first observe that the closed cone $\Omega$ of $X$ belongs to $\CcC_{\sigma}$
and therefore $X \times W(\Omega,*) \in \CcC_{\sigma}$ as well (since $\CcC$ is product). Further, (SU4) yields that
$X \times W(\Omega,*)$ is $\CcC$-absorbing. So, $X$ and $X \times W(\Omega,*)$ are homeomorphic (thanks to (SU5)).
Now use \THM{pre-H} and Toru\'{n}czyk's Factor Theorem \cite{tor4,tor3} to finish the proof.\par
We pass to showing that (ii) follows from (iii). Let $\CcC$ be the class of all topological spaces which are closed
embeddable in $\Omega$. Clearly, $\CcC$ is topological and closed hereditary. By hypothesis, $X \times \Omega \cong X$
and thus also $X \times \Omega \times \Omega \cong X$. This yields that
\begin{equation}\label{eqn:aux10}
\Omega \times \Omega \text{ is closed embeddable in } \Omega
\end{equation}
(note that $X \in \CcC$). Moreover, $\Omega$ is noncompact because $X$ is a $\sigma$-$Z$-space. So, we infer
from \PRO{cl-emb} that $\CcC$ is open hereditary and additive. Noticing that $\CcC$ is also product
thanks to \eqref{eqn:aux10}, we see that to finish the proof it suffices to show that $X$ is $\CcC$-absorbing. We shall do
this using (SU2).\par
Let $K$ be a $Z$-set in $X$. For simplicity, put $U = X \setminus K$ and let $\pi\dd U \times \Omega \times \Omega \to U$
be the natural projection. It follows from the assumptions that $\pi$ is a near-homeomorphism. This, combined
with \LEM{omega}, yields that the natural projection $U \times w(U) \to U$ is approximable by closed embeddings
and therefore $U$ is homotopy dense embeddable in a Hilbert manifold (thanks to Banakh's theorem \cite{banakh}).
Fix any $x_0 \in X$ and put $a = (x_0,1) \in \Omega$. The set $\{a\}$ is a $Z$-set in $\Omega$ (because $\{x_0\}$ is
a $Z$-set in $X$, which is implied by the fact that $X$ is a $\sigma$-$Z$-space). Now take a space $C \in \CcC$ and a map
$f\dd C \to U$. There is a closed embedding $u\dd C \to \Omega$. We see that $v\dd C \ni x \mapsto (f(x),u(x),a) \in U
\times \Omega \times \Omega$ is a $Z$-embedding such that $\pi \circ v = f$. So, since $\pi$ is a near-homeomorphism,
$f$ is approximable by $Z$-embeddings. This shows that $U$ is $\CcC$-universal. Hence an application of (SU2) finishes
the proof.
\end{proof}

We do not know whether in the above result it suffices to check the point (iii) only for $K = \varempty$
to obtain (i).\par
The technique of rigid embeddings enables us to give a simple proof of the following generalization of a special case
of \cite[Theorem 2.4.2]{brz}:

\begin{pro}{product}
Let $\CcC$ be a topological, closed hereditary, additive and product class of metrizable spaces such that $I \in \CcC$.
There is a $\CcC$-absorbing pre-Hilbert space $F$ with $F \cong \Sigma F$ iff there is a space $X$ such that
\begin{enumerate}[\upshape(U1)]
\item $X \in \CcC_{\sigma}$ and
\item every member of $\CcC$ is closed embeddable in $X$.
\end{enumerate}
\end{pro}
\begin{proof}
The necessity is clear. To prove the sufficiency, embed rigidly the space $X$ into a pre-Hilbert space $E$ in such a way
that $E$ coincides with the linear span of the image of the embedding and then apply \PRO{c-sigma} and (SU7).
\end{proof}

\begin{exm}{fd}
For a countable ordinal $\alpha > 0$ and an infinite cardinal $\mM$ let $\MmM_{\alpha}(\mM)$ [$\MmM_{\alpha}^f(\mM)$]
and $\AaA_{\alpha}(\mM)$ [$\AaA_{\alpha}^f(\mM)$] be the class of all (metrizable) [finite dimensional] spaces
of the absolute multiplicative and additive (respectively) Borelian class $\alpha$ and of weight no greater than $\mM$
(for definition and more on these classes see e.g. \cite{stone1,stone2}, \cite{hansell} or \cite{kur}, where the Reader
can find more references concerning the subject). It is easily seen that each of these classes is topological, closed
hereditary, additive and product. Bestvina and Mogilski \cite{b-m} have proved that for an arbitrary $\alpha$, there are
an $\MmM_{\alpha}(\aleph_0)$-absorbing AR and an $\AaA_{\alpha}(\aleph_0)$-absorbing one, and Banakh, Radul and Zarichnyi
have shown that also the classes $\MmM_{\alpha}^f(\aleph_0)$ and $\AaA_{\alpha}(\aleph_0)$ have absorbing AR's (see
\cite[Theorem~2.4.9]{brz}). Note that the classes $\MmM_1$ and $\AaA_1$ consist of \textit{absolute
$\GGg_{\delta}$-spaces} and \textit{absolute $\FFf_{\sigma}$-spaces}, respectively. It is well known that a metrizable
space is of absolute $\GGg_{\delta}$-class iff it is completely metrizable, and such a space is of absolute
$\FFf_{\sigma}$-class iff it is $\sigma$-locally compact (\cite{stone0}), that is, it is the countable union of its
locally compact subsets or, equivalenty, the countable union of its locally compact closed subsets. This easily implies
that $\Sigma l^2(\mM)$ is an $\MmM_1(\mM)$-absorbing AR for each infinite cardinal $\mM$ (where $l^2(\mM)$ is the Hilbert
space with an orthonormal basis of cardinality $\mM$), which was obtained e.g. by Toru\'{n}czyk \cite{tor?}. Further,
in \cite{tsuda} Tsuda has proved that for each natural $n$ and infinite cardinal $\mM$ there is a completely metrizable
finite dimensional space $T_n(\mM)$ of weight $\mM$ such that every completely metrizable space $X$ with $\dim X \leqsl n$
and $w(X) \leqsl \mM$ is closed embeddable in $T_n(\mM)$. This implies that the topological disjoint union
$T(\mM) = \bigoplus_{n=1}^{\infty} T_n(\mM)$ of the spaces $T_n(\mM)$ contains a closed copy of every member
of $\CcC = \MmM_1^f(\mM)$ and belongs to $\CcC_{\sigma}$. We conclude from this and \PRO{product} that $\MmM_1^f(\mM)$
has an absorbing AR. Analogous results hold for the classes $\AaA_1$: the space $l^2_f(\mM) =$ the linear span of a fixed
orthonormal basis of $l^2(\mM)$ is an $\AaA_1^f(\mM)$-absorbing AR (cf. \cite{west}). Indeed, it is clear that
$\Sigma l^2_f(\mM) \cong l^2_f(\mM)$ and that $I^n \times \mM$ is closed embeddable in $l^2_f(\mM)$ for each natural $n$.
Since every finite dimensional locally compact metrizable space $X$ of weight no greater than $\mM$ is the union of its
two closed subsets each of which is closed embeddable in $I^n \times \mM$ (for some $n$), thus each such a space $X$
is closed embeddable in $l^2_f(\mM)$ and therefore---by \PRO{sigma}---$\AaA_1^f(\mM) \subset \FfF(l^2_f(\mM))$.
On the other hand, if $Y$ is the closed cone over the topological disjoint union of the spaces $I^n \times \mM$
($n=1,2,\ldots$ and $\mM$ is fixed), then $Y$ is $\sigma$-finite dimensional (i.e. $Y$ is the countable union of its
closed finite dimensional subsets) and $\sigma$-locally compact and $W(Y,*)$ is therefore an $\AaA_1^f(\mM)$-absorbing AR.
So, (SU5) gives $W(Y,*) \cong l^2_f(\mM)$. The latter method of finding an $\AaA_1^f(\mM)$-absorbing AR works also
for the classes $\AaA_1(\mM)$: it suffices to take $W(Z,*)$ where $Z$ is the closed cone over $Q \times \mM$.\par
For ordinals $\alpha$ greater than $1$ (and uncountable cardinals $\mM$) and $\CcC \in \{\MmM_{\alpha}(\mM),
\MmM_{\alpha}^f(\mM), \AaA_{\alpha}(\mM), \AaA_{\alpha}^f(\mM)\}$ there are $\CcC$-absorbing sets if only there are spaces
$X$ satisfying (U1) and (U2) (thanks to \PRO{product}). The author however has no knowledge whether such spaces exist.
\end{exm}

We end the paper with the following

\begin{pro}{iso-rigid}
Every nonempty metric space is isometric to a subset $A$ of some normed linear space $(E,\|\cdot\|)$ such that
for each $n \geqsl 1$, $t \in \RRR^n$ and $a \in A^n$,
\begin{equation}\label{eqn:l1}
\|t \bullet a\| \geqsl \min(\frac12 \delta_n(a),1) |t|_1,
\end{equation}
and if the metric of the given space is upper bounded by $2$, $A$ may be taken to be contained in the unit sphere of $E$.
\end{pro}
\begin{proof}
We shall improve Michael's proof \cite{michael2} (cf. \cite[Theorem II.1.2]{b-p}) of the Arens-Eells theorem \cite{a-e}.
Let $(X,d)$ be a metric space and $x_*$ an arbitrarily chosen element of $X$. If $d$ is upper bounded by $2$, put
$\psi \equiv 1\dd X \to [1,+\infty)$; otherwise let $\psi\dd X \ni x \mapsto d(x,x_*) + 1 \in [1,+\infty)$. Take $\omega
\notin X$. We extend the metric $d$ to a metric on the set $\tilde{X} = X \cup \{\omega\}$ by putting $d(x,\omega)
= \psi(x)$ for $x \in X$. Now let $\Gamma$ consists of all $d$-nonexpansive maps of $\tilde{X}$ into $\RRR$ which vanish
at $\omega$. That is, $u\dd \tilde{X} \to \RRR$ belongs to $\Gamma$ iff $|u(z_1) - u(z_2)| \leqsl d(z_1,z_2)$ for all
$z_1,z_2 \in \tilde{X}$ and $u(\omega) = 0$. Let $l^{\infty}(\Gamma)$ be the Banach space of all real-valued bounded
functions on $\Gamma$ equipped with the supremum norm. For $x \in X$ denote by $\delta_x \in l^{\infty}(\Gamma)$
the evaluation map of $x$, i.e. $\delta_x(u) = u(x)$ for $u \in \Gamma$. It is easily seen that the map $h\dd (X,d) \ni x
\mapsto \delta_x \in (l^{\infty}(\Gamma),\|\cdot\|)$ is isometric and that $\im h$ is contained in the unit sphere
of $l^{\infty}(\Gamma)$ if $d$ is upper bounded by $2$. We shall check that \eqref{eqn:l1} holds.\par
Let $t = (t_1,\ldots,t_n) \in \RRR^n$ and $x = (x_1,\ldots,x_n) \in X^n$ be such that $\delta_n(x) \geqsl r \in (0,2]$.
Observe that the map $u_0\dd \{x_1,\ldots,x_n\} \to \RRR$ given by $u_0(x_j) = \frac{r}{2}\sgn t_j$ (where $\sgn$ is
the sign function) is $d$-nonexpansive and takes values in $[-1,1]$. This yields that there is a map $u \in \Gamma$
which extends $u_0$. But then $\|\sum_{j=1}^n t_j \delta_{x_j}\| \geqsl |\sum_{j=1}^n t_j u(x_j)| = \frac{r}{2}|t|_1$.
This, together with the fact that $h$ is isometric, gives \eqref{eqn:l1}.
\end{proof}

Having in mind \EXM{H} and \PRO{iso-rigid}, the following may be an interesting question:\par
\textit{Is every metrizable space homeomorphic to a rigid subset of the unit sphere of some Hilbert space?}

\end{document}